\documentclass[12pt]{article}
\usepackage[amsmath]{e-jc}
\usepackage[T1]{fontenc}   % needed for đ in "Ranđelović"
\usepackage{graphicx}
\usepackage{subcaption}
\usepackage{booktabs}
\usepackage{array}
\usepackage{enumitem}

\theoremstyle{definition}
\newtheorem{invariant}{Invariant}

\theoremstyle{plain}
\newtheorem*{theoremA}{Theorem A}

\numberwithin{equation}{section}


\Copyright{The author. Released under the CC BY-ND license (International 4.0).}

\title{The transversal achievement game\\ on a square grid}

\author{Kevin Guan}
\date{}
\authortext{}{Department of Mathematics, Princeton University, Princeton,
   New Jersey, USA (\email{kevinguan@princeton.edu}).}

\begin{document}
\maketitle

\begin{abstract}
In the transversal achievement game on the $n\times n$ board, two players alternately claim cells, and the first to own a transversal---a set of $n$ cells of which no two share a row or column---wins. Ranđelović showed that the first player wins for every $n\ge4$, while the game is a draw for $n=2,3$. We give an independent proof that the first player wins for $n\ge4$ that additionally establishes a bound on the length of the win: the given strategy forces a win by ply $2n+3$, i.e.\ on the first player's $(n+2)$-nd move, for every $n\ge4$. The proof yields a strategy that is fully determined by a fixed rule on the current position and can thus be implemented directly. We isolate the use of the hypothesis $n\ge4$ to two steps in the analysis, explaining why the argument fails at $n=3$. An exhaustive computational search implementing the strategy verifies it against every legal defense for $n=4,5,6$, confirming both the strategy's validity and that the $2n+3$ bound is attained in these cases. The main theorem has also been formalized and machine-checked in Lean 4.
\end{abstract}

%---------------------------------------------------------------
\section{Introduction}
\label{sec:intro}
%---------------------------------------------------------------

Two players alternately claim cells of an $n\times n$ grid, and the first to own a \emph{transversal}---a set of $n$ cells with no two cells sharing a row or a column--- wins. If the board fills with neither player having done so, the game is a draw. This is the \textbf{transversal achievement game}, introduced and posed as an open problem by Erickson~\cite{erickson}. 
\begin{figure}[ht]
\centering
\begin{subfigure}[b]{0.2829\textwidth}
\centering
\includegraphics[width=\textwidth]{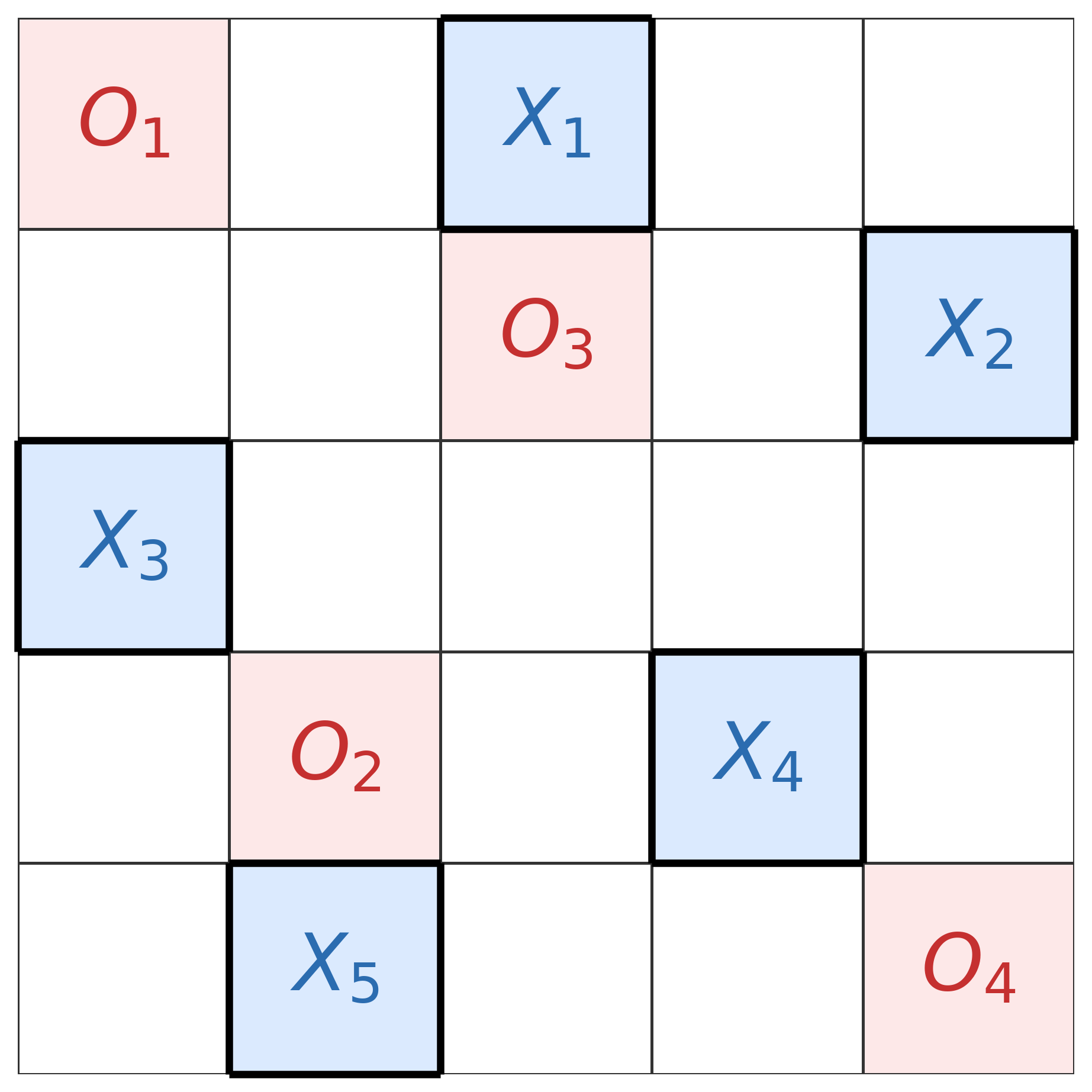}
\caption{Player 1 (X) win}
\label{fig:example-x-win}
\end{subfigure}
\hfill
\begin{subfigure}[b]{0.2829\textwidth}
\centering
\includegraphics[width=\textwidth]{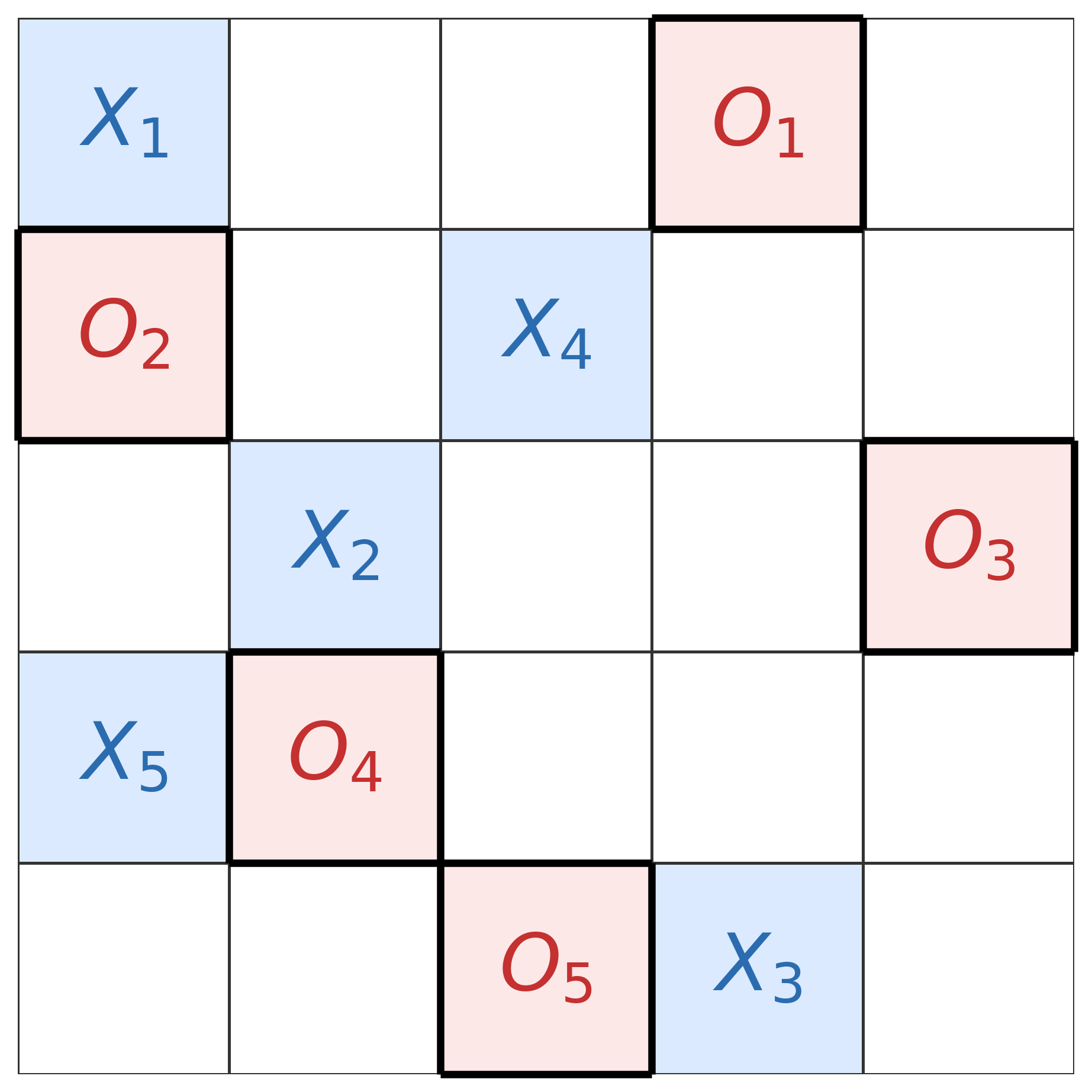}
\caption{Player 2 (O) win}
\label{fig:example-o-win}
\end{subfigure}
\hfill
\begin{subfigure}[b]{0.2829\textwidth}
\centering
\includegraphics[width=\textwidth]{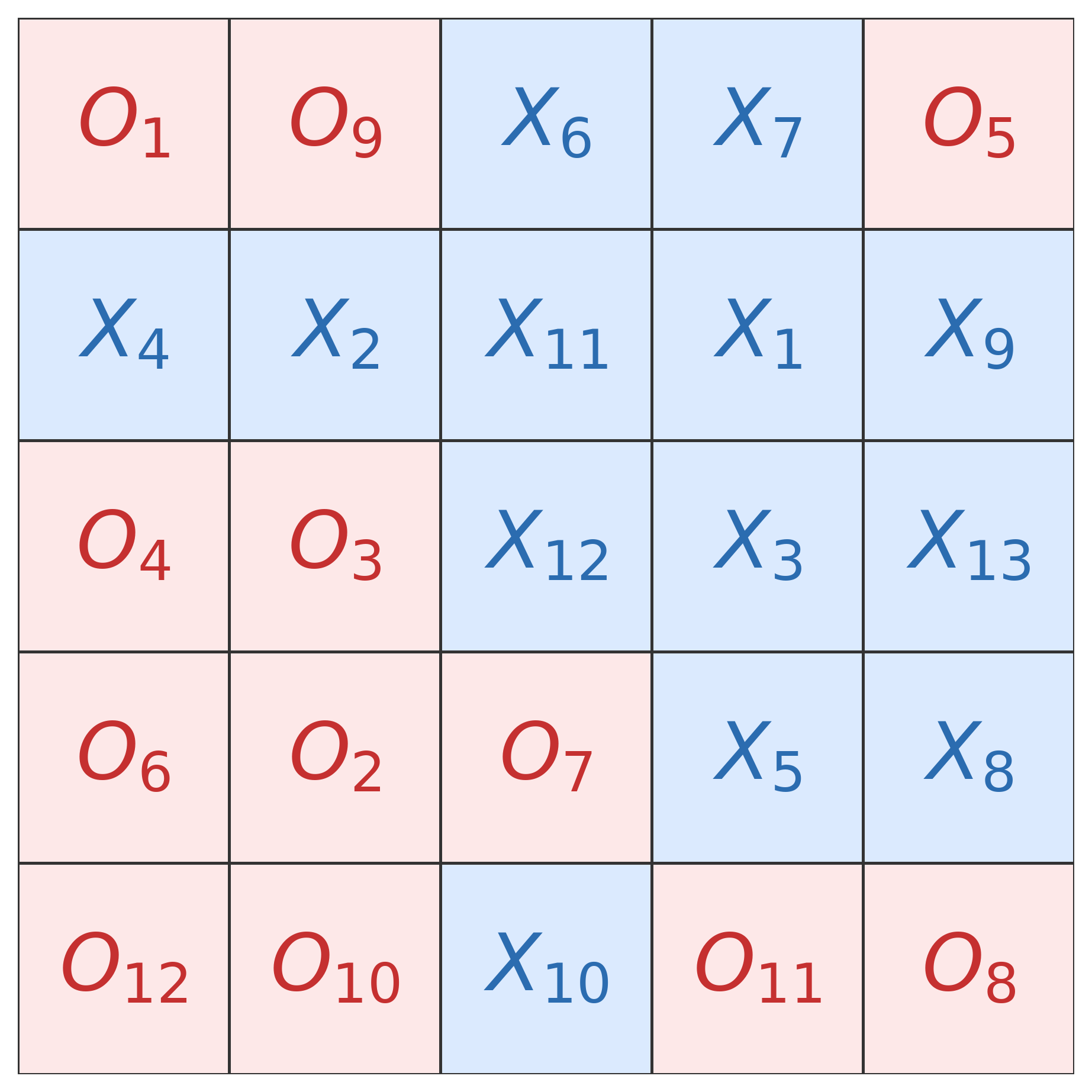}
\caption{Draw}
\label{fig:example-draw}
\end{subfigure}
\caption{Examples of the three possible outcomes on a $5\times5$ board.}
\label{fig:examples}
\end{figure}
It is a positional game in the sense of Beck~\cite{beck}: the board is the set of $n^2$ cells and the winning family is the set of $n!$ transversals, equivalently the perfect matchings of $K_{n,n}$. Unlike tic-tac-toe--style games, the winning sets are enormous in number and heavily overlapping. The Erd\H{o}s--Selfridge criterion~\cite{erdos-selfridge} gives no useful bound here, since $\sum_{T}2^{-|T|}=n!\,2^{-n}$ is astronomically larger than $\tfrac12$; pairing strategies are unavailable for the same reason.

For $n=1$, the first player wins trivially, and for $n=2$, the game is a draw; the standard strategy-stealing argument (as in Hales--Jewett~\cite{hales-jewett}) shows that the second player never wins, so for each $n$ the value is either \emph{first-player win} or \emph{draw}. The $3\times3$ game is similarly a draw, and Ranđelović~\cite{randjelovic} proved that the first player wins for every $n\ge4$, thereby answering Erickson's question. That paper also proves the $n=3$ draw by hand and goes on to study how small a sub-family of transversals still suffices for a first-player win.

The present work gives an independent proof that the first player wins for $n\ge4$, with several features that distinguish it from the argument of~\cite{randjelovic}:
\begin{enumerate}[leftmargin=2em]
\item \textbf{A bound on the length of the win.} We show that the strategy we give wins by ply $2n+3$---with the first player's $(n+2)$-nd stone---for every $n\ge4$, and the search reported in \S\ref{sec:verification} finds that the bound is attained by this strategy at $n=4,5,6$. 

\item \textbf{A fixed frame with matching theory.} The proof of~\cite{randjelovic} keeps its position in normal form by permuting rows and columns after almost every move. Instead of moving the board, we maintain an invariant---the \emph{open block} $H=U_R\times U_C$ contains no opposing stone (\S\ref{sec:strategy}). The two lemmas of \S\ref{sec:lemmas} then accomplish what the case analysis accomplishes in~\cite{randjelovic}: Lemma~\ref{lem:threat} identifies the cells completing a set $S$ with $\nu(S)=n-1$ as a combinatorial rectangle $D_R\times D_C$, and Lemma~\ref{lem:rectangle} describes how that rectangle grows when a stone is added to row $b$ or column $d$.

\item \textbf{One meaningful case split.} Our analysis only branches on the parameter $w=|F\cap(\text{row }b\cup\text{col }d)|$, the number of the opponent's early stones lying in the critical line pair, giving three cases (\S\ref{sec:proof}). By comparison, the argument of~\cite{randjelovic} handles $n=4$ separately from $n\geq 5$ and proceeds through five terminal cases, with a three-way split in the inductive construction preceding them.

\item \textbf{Explanation of the threshold.} We isolate the use of $n\ge4$ to exactly two places, named in \S\ref{sec:proof}. Everything else survives at $n=3$, which identifies why the argument begins to work at $n=4$.

\item \textbf{Machine-checked formalization.} The main theorem, $n=3$
draw, and supporting lemmas are formalized in Lean~4 and checked by the
Lean kernel (Appendix~\ref{app:lean}).
\end{enumerate}
Another more practical difference is that \S\ref{sec:strategy} specifies the strategy completely and separately from the verification in \S\ref{sec:proof}. It can therefore be implemented and run without reference to the proof, which makes the exhaustive checks in \S\ref{sec:verification}, of the strategy against every defense for $n=4,5,6$, meaningful as a corroboration.

\medskip

\noindent\textbf{Organization.} \S\ref{sec:prelim} fixes the game, the identification of the board with $K_{n,n}$, and the language of threats and completing cells; it states Theorem~A and disposes of the cases $n\le3$. \S\ref{sec:lemmas} proves two structural lemmas on which the main proof rests. \S\ref{sec:strategy} specifies X's strategy, and \S\ref{sec:proof} proves that X wins by ply $2n+3$. \S\ref{sec:verification} reports the exhaustive computational
verification for $n=4,5,6$, and \S\ref{sec:discussion} closes with
open problems. Appendices~\ref{app:implementation}--\ref{app:playthrough}
give implementation details, search statistics, a Lean~4 formalization of
the main theorem, and a worked $6\times6$ playthrough.

%---------------------------------------------------------------
\section{Preliminaries and Main Result}
\label{sec:prelim}
%---------------------------------------------------------------

We write the first player as \textbf{X} and the second as \textbf{O}.

\medskip

\noindent\textbf{The game.} Two players alternately claim cells of an $n\times n$ grid; Player 1 (X) moves first. A player wins upon owning $n$ cells of which no two share a row or a column (a \emph{transversal}, equivalently a perfect matching of $K_{n,n}$). If neither player ever does, the game is a draw. X's $k$-th stone is placed at ply $2k-1$, and O's $k$-th stone is placed at ply $2k$.

\medskip

\noindent\textbf{The board as $K_{n,n}$.} We denote the rows of the board as $R$ and the columns as $C$, with $|R|=|C|=n$. Cells are pairs $(r,c)$. We identify the board with the complete bipartite graph $K_{n,n}$ whose
vertex classes are the rows $R$ and columns $C$. Under this identification,
a cell $(r,c)$ is the edge joining row $r$ to column
$c$. Thus a matching of cells is
exactly a graph matching, i.e., a collection of edges of which no two share a
row vertex or a column vertex. In particular, a transversal is a perfect
matching of this bipartite graph.

\medskip

\noindent\textbf{Matchings, threats, and completing cells.} For a set $S$ of cells, $\nu(S)$ is the maximum size of a matching inside $S$, i.e., a set with no two cells sharing a row or column. A player wins exactly when their set $S$ satisfies $\nu(S)=n$, so ``$S$ contains a transversal'' reads $\nu(S)=n$. We say a cell $f$ \textbf{completes} $S$ if $\nu(S\cup\{f\})=n$, and that a player \textbf{threatens} $f$ if $f$ is unoccupied and completes their set. Note that a player may have $\nu(S)=n-1$ and still threaten nothing, because every completing cell is already occupied. This situation occurs in case (C-iii) of the proof of Lemma~\ref{lem:nodeviation}.

\begin{theoremA}
Player 1 wins for $n=1$ and for every $n\ge 4$; $n=2,3$ are draws.
\end{theoremA}

\medskip
Sections~\ref{sec:lemmas}--\ref{sec:proof} prove the case $n\ge4$ with the strategy of \S\ref{sec:strategy}: X wins by ply $2n+3$, i.e.\ with their $(n+2)$-nd stone. The cases $n<4$ are settled below and by previous work: $n=1$ is trivial, and $n=2$ is settled by hand. We do not reprove that $n=3$ is a draw~\cite{randjelovic}, though it is included in the Lean development of Appendix~\ref{app:lean}. We show through an exhaustive search that the $2n+3$ upper bound is attained at $n=4,5,6$ (\S\ref{sec:verification}).

\medskip
\noindent\textbf{Player 2 never wins (any $n$).} The winning family is monotone, as a superset of a transversal contains a transversal, and the board is finite, so a strategy-stealing argument applies. Assume O has a winning strategy $\Sigma$. X could adopt $\Sigma$ after an arbitrary first move, treating their extra stone as unplayed and playing an arbitrary cell when $\Sigma$ names an occupied one. Compare the real game with the imagined game obtained by deleting X's extra stones: O's set is identical in both, and X's is a superset, so if X wins the imagined game at ply $t$, then in the real game, O has not completed a transversal before ply $t$ and X completes one no later. This yields a win for X, contradicting the fact that both cannot win. Hence, for every $n$, the game value is \emph{X win} or \emph{draw}; ruling out an X win at $n=2,3$ is what makes those draws.

\medskip
\noindent\textbf{The small cases $n\le3$.} For $n=1$ the first player trivially wins. For $n=2$ the two winning sets are the diagonals $A=\{(1,1),(2,2)\}$ and $B=\{(1,2),(2,1)\}$. O can draw with the strategy of answering X's first stone with its partner in the same diagonal. That diagonal is then split one--one, and the remaining two cells are exactly the other diagonal, which X and O take one each. Neither owns a diagonal, so O's strategy prevents an X win. Since O cannot win by the strategy-stealing remark, $n=2$ is a draw. The $3\times3$ game is likewise a draw, proved by hand in~\cite{randjelovic}. The rest of the paper treats $n\ge4$.

%---------------------------------------------------------------
\section{Structural Lemmas}
\label{sec:lemmas}
%---------------------------------------------------------------

\begin{lemma}[threat structure]
\label{lem:threat}
Let $\nu(S)=n-1$. Let $D_R$ be the set of rows exposed by some maximum matching of $S$, and $D_C$ the set of columns exposed by some maximum matching. Then
\[
\{f:\ f \text{ completes } S\}=D_R\times D_C .
\]
Additionally, if $\nu(S)\le n-2$, no single cell completes $S$.
\end{lemma}

\begin{proof}
Since $\nu(S)=n-1$, every maximum matching exposes exactly one row and one column.

$(\subseteq)$ If $\nu(S\cup\{(p,q)\})=n$, a perfect matching of $S\cup\{(p,q)\}$ must use $(p,q)$, as otherwise $\nu(S)=n$. Deleting it leaves a maximum matching of $S$ exposing $p$ and $q$, so $p\in D_R,\ q\in D_C$.

$(\supseteq)$ Let $M_1$ be a maximum matching exposing row $p$ (and some column $q_1$), and $M_2$ a maximum matching exposing column $q$ (and some row $p_2$). We construct a maximum matching that exposes $p$ and $q$ simultaneously, as adjoining $(p,q)$ then gives a perfect matching. If $p_2=p$ take $M_2$; if $q_1=q$ take $M_1$. Otherwise consider the component $P$ of $M_1\triangle M_2$ containing $p$. As $p$ is $M_1$-exposed and $M_2$-covered, it has degree $1$ in $M_1\triangle M_2$, so $P$ is a path starting at $p$ with an $M_2$-edge. Its other endpoint $z$ also has degree $1$. If the last edge of $P$ were an $M_2$-edge, $z$ would be $M_1$-exposed and $P$ would be an $M_1$-augmenting path, contradicting maximality of $M_1$; so the last edge lies in $M_1$ and $z$ is $M_2$-exposed. The path alternates sides of the bipartition and starts at the row vertex $p$ with an $M_2$-edge and ends with an $M_1$-edge, so it has evenly many edges and $z$ is a row vertex; the only $M_2$-exposed row vertex is $p_2$, so $z=p_2$. Then $|P\cap M_1|=|P\cap M_2|$, so $M_2\triangle P$ is again a maximum matching, which exposes $p$, whose only incident $M_2$-edge was removed, and covers $p_2$. Finally, $q\notin V(P)$: internal vertices of $P$ have degree $2$ in $M_1\triangle M_2$ and are thus covered by both matchings, while $q$ is $M_2$-exposed, and the endpoints are $p\neq q$ and $p_2\ne q$, as both are row vertices. Therefore, $M_2\triangle P$ still exposes $q$.
\end{proof}

\noindent\emph{Two remarks on the statement.} First, $D_R\times D_C$ is automatically disjoint from $S$: if $(p,q)\in S$ with $p\in D_R$, $q\in D_C$, the matching produced above together with $(p,q)$ would be a perfect matching \emph{inside} $S$, contradicting $\nu(S)=n-1$. So ``completes'' is well behaved, and no cell of the displayed set is already occupied by the owner of $S$. Second, the final clause holds because adding one cell raises $\nu$ by at most $1$.

\begin{corollary}[tempo]
\label{cor:tempo}
A player holding fewer than $n-1$ stones has $\nu\le n-2$ and hence no threat; a player holding fewer than $n$ stones has not won. If a player's set contains a matching $M$ of size $n-1$ missing row $b$ and column $d$ and no other cells, then $D_R=\{b\}$, $D_C=\{d\}$, and $(b,d)$ is the unique completing cell.
\end{corollary}

Corollary~\ref{cor:tempo} says that the first player to assemble an $(n-1)$-matching can complete it, while the opponent---still short of $n-1$ stones---has no threat at all. It does not by itself imply that the completing cell is free; that is supplied by the tie-break of \S\ref{sec:strategy}, which preserves Invariant~\ref{inv:openblock} through X's last Phase-1 move.

\begin{lemma}[growth of the threat rectangle]
\label{lem:rectangle}
Let $M$ be a matching of size $n-1$ missing row $b$ and column $d$, with induced bijection $\sigma:R\setminus\{b\}\to C\setminus\{d\}$. Let $r\ne s$ in $R\setminus\{b\}$. In each case below, $\nu(S)=n-1$, with:
\begin{center}
\begin{tabular}{@{}llll@{}}
\toprule
$S$ & $D_R$ & $D_C$ & completing cells\\
\midrule
(a) $M\cup\{(b,\sigma(r))\}$
& $\{b,r\}$ & $\{d\}$
& $(b,d),(r,d)$\\
(b) $M\cup\{(r,d)\}$
& $\{b\}$ & $\{d,\sigma(r)\}$
& $(b,d),(b,\sigma(r))$\\
(c) $M\cup\{(b,\sigma(r)),(s,d)\}$
& $\{b,r\}$ & $\{d,\sigma(s)\}$
& $\{b,r\}\times\{d,\sigma(s)\}$\\
(d) $M\cup\{(r,d),(b,\sigma(s))\}$
& $\{b,s\}$ & $\{d,\sigma(r)\}$
& $\{b,s\}\times\{d,\sigma(r)\}$\\
\bottomrule
\end{tabular}
\end{center}
\end{lemma}

The four cases are summarized visually in Figure~\ref{fig:lemma3}; the shaded
rectangle in each panel is the resulting completing set
$D_R\times D_C$.

\begin{figure}[ht]
\centering
\includegraphics[width=\textwidth]{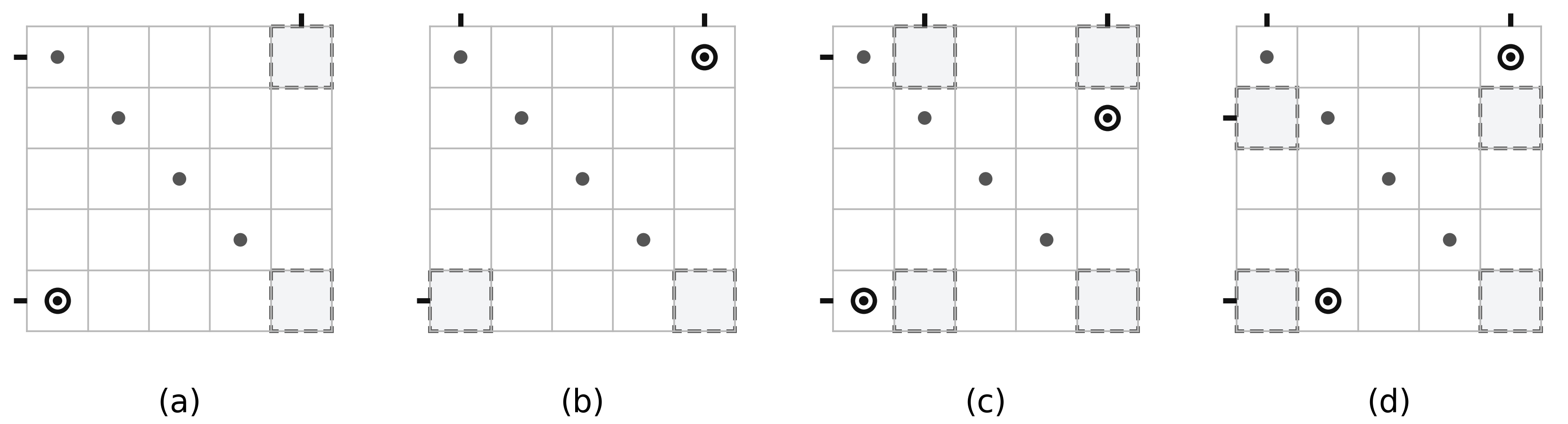}
\caption{The configurations in Lemma~\ref{lem:rectangle} for an $(n-1)$-matching
$M=\{(i,i):1\le i\le4\}$, with
$b=d=5$, $\sigma=\mathrm{id}$, $r=1$, and $s=2$. Filled dots are existing edges of $M$; the outlined dots are newly added cells. The dashed rectangle is the resulting completing set
$D_R\times D_C$. For example, in (a), adding $(b,\sigma(r))$ enlarges the exposed rows
from $\{b\}$ to $\{b,r\}$, giving completing cells
$\{b,r\}\times\{d\}$.}
\label{fig:lemma3}
\end{figure}
\begin{proof}
In each case $\nu(S)<n$: a perfect matching would have to cover row $b$, whose only cell in $S$ is $(b,\sigma(r))$ in (a),(c) and $(b,\sigma(s))$ in (d)---after which row $r$ (resp.\ $s$) has only $(r,\sigma(r))$ (resp.\ $(s,\sigma(s))$) available, whose column is already used; in (b) row $b$ is empty. So $\nu(S)=n-1$, witnessed by $M$.

For the exposed sets, it is cleanest to decompose $S$ into components. In (c) these are $n-3$ isolated edges $(i,\sigma(i))$, $i\notin\{b,r,s\}$, together with the paths $b-\sigma(r)-r$ and $d-s-\sigma(s)$. A maximum matching takes all forced edges and one edge from each path, giving the four matchings $M$, $M-(r,\sigma(r))+(b,\sigma(r))$, $M-(s,\sigma(s))+(s,d)$ and $M-(r,\sigma(r))-(s,\sigma(s))+(b,\sigma(r))+(s,d)$, exposing the pairs $(b,d),(r,d),(b,\sigma(s)),(r,\sigma(s))$ respectively; hence $D_R=\{b,r\}$, $D_C=\{d,\sigma(s)\}$, and Lemma~\ref{lem:threat} gives the product. Cases (a),(b) drop one path, and (d) is the mirror of (c).
\end{proof}

\begin{remark}
Only cells in row $b$ or column $d$ create completing cells: adding $(p,q)$ with $p\ne b,\ q\ne d$ to $M$ creates the $3$-edge path $\sigma(p)-p-q-\sigma^{-1}(q)$, whose unique maximum matching consists of the two $M$-edges; so $M$ remains the unique maximum matching, and  $D_R=\{b\},D_C=\{d\}$.
\end{remark}

%---------------------------------------------------------------
\section{The Winning Strategy}
\label{sec:strategy}
%---------------------------------------------------------------

This section states X's strategy in full, from the opening to the final double threat; \S\ref{sec:proof} proves that every choice it calls for is available and that it wins.

\subsection*{Phase 1: building the $(n-1)$-matching (X's moves $1,\dots,n-1$)}

Let $U_R,U_C$ be the rows and columns, respectively, containing no X-stone, and let the \textbf{open block} be $H=U_R\times U_C$. X will keep their stones a matching inside the successive open blocks, so $|U_R|=|U_C|=n-k$ after their $k$-th move.

\begin{invariant}
\label{inv:openblock}
$H$ contains no O-stone.
\end{invariant}

\noindent\textbf{X's Phase-1 rule (moves $1,\dots,n-2$).} Move $1$: any cell. Move $k$ ($2\le k\le n-2$): let $x$ be O's most recent stone. If $x\in H$, play a free cell of $H$ in $x$'s row; otherwise play any free cell of $H$. Move $n-1$ is governed by the tie-break below. Lemma~\ref{lem:phase1} shows that this rule is always executable and that Invariant~\ref{inv:openblock} holds after each of X's moves.

Note also that O can neither win nor threaten during plies $1,\dots,2n-4$: they then hold at most $n-2$ stones, so Corollary~\ref{cor:tempo} applies. X is therefore never obliged to answer anything in Phase 1.

\subsection*{The tie-break at X's move $n-1$}

Just before move $n-1$ we have $|U_R|=|U_C|=2$, say $H=\{u_1,u_2\}\times\{v_1,v_2\}$, and X's $n-2$ stones form a perfect matching $M_0$ of $A\times B$, where $A=R\setminus\{u_1,u_2\}$, $B=C\setminus\{v_1,v_2\}$. If X plays $(u_a,v_c)$ then $(b,d)=(u_{3-a},v_{3-c})$. We call a cell of $H$ \textbf{admissible} if it is free and its opposite corner is free. This is precisely the condition that keeps $(b,d)$ free, i.e., that preserves Invariant~\ref{inv:openblock}. Writing $F$ for the $n-2$ O-stones already placed, each candidate outcome $(b,d)$ carries the parameter
\[
w:=\bigl|F\cap(\text{row }b\cup\text{column }d)\bigr| .
\]
\noindent\textbf{Existence of an admissible cell.} By Invariant~\ref{inv:openblock} and Lemma~\ref{lem:phase1}, $H$ contains no X-stone and at most one O-stone, namely O's latest $x_{n-2}$.

\begin{itemize}[leftmargin=2em]
\item If $x_{n-2}=(u_i,v_j)\in H$: three cells are free, of which $(u_{3-i},v_{3-j})$ is inadmissible, since its opposite is $x_{n-2}$, and the remaining two, $(u_{3-i},v_j)$ and $(u_i,v_{3-j})$, are admissible. They give the two possible outcomes $(b,d)=(u_i,v_{3-j})$---then $x_{n-2}$ lies in row $b$---and $(b,d)=(u_{3-i},v_j)$---then $x_{n-2}$ lies in column $d$. Either way, $w\ge1$.
\item If $x_{n-2}\notin H$: all four cells are free, and hence admissible and possible values of $(b,d)$.
\end{itemize}

\begin{quote}
\textbf{X's tie-break.} Choose an admissible outcome with $1\le w\le n-3$ if one exists; otherwise choose any admissible outcome with $w\le n-3$ (one exists by Lemma~\ref{lem:tiebreak}, and it then necessarily has $w=0$).
\end{quote}

In particular, the tie-break always delivers $w\le n-3$.

\medskip

\noindent\textbf{Consequence.} After move $n-1$ (ply $2n-3$), $|U_R|=|U_C|=1$; write $U_R=\{b\}$, $U_C=\{d\}$. Admissibility keeps $(b,d)$ free, while X owns exactly a matching $M$ of size $n-1$ missing row $b$ and column $d$.

\medskip

By Corollary~\ref{cor:tempo}, X threatens $(b,d)$ at ply $2n-3$, whereas O, holding $n-2$ stones, has no threat and has not won. So O must play $(b,d)$ at ply $2n-2$ or lose at ply $2n-1$. From ply $2n-2$ on, O is forced, subject only to the verification (Lemma~\ref{lem:nodeviation}) that they never acquire a threat or a win.

\subsection*{Phase 2: the two plans}

Write $\sigma:R\setminus\{b\}\to C\setminus\{d\}$ for the bijection induced by $M$. Call $s\in R\setminus\{b\}$ \textbf{live} if $(s,d)$ and $(b,\sigma(s))$ are both free at ply $2n-3$ (Figure~\ref{fig:two-live-rows}). Since $M$ misses row $b$ and column $d$, X owns none of the $2(n-1)$ distinct cells
\[
\{(s,d):s\ne b\}\ \cup\ \{(b,\sigma(s)):s\ne b\}\ =\ \bigl(\text{row }b\cup\text{col }d\bigr)\setminus\{(b,d)\},
\]
and $(b,d)$ is free, so the only occupied ones are the $w$ cells of $F$ there; each kills at most one row $s$. Hence
\begin{equation}
\label{eq:live}
\ell:=\#\{\text{live rows}\}\ \ge\ (n-1)-w.
\end{equation}

\begin{figure}[htbp]
\centering

\begin{subfigure}[t]{0.43\textwidth}
    \centering
    \includegraphics[width=\textwidth]{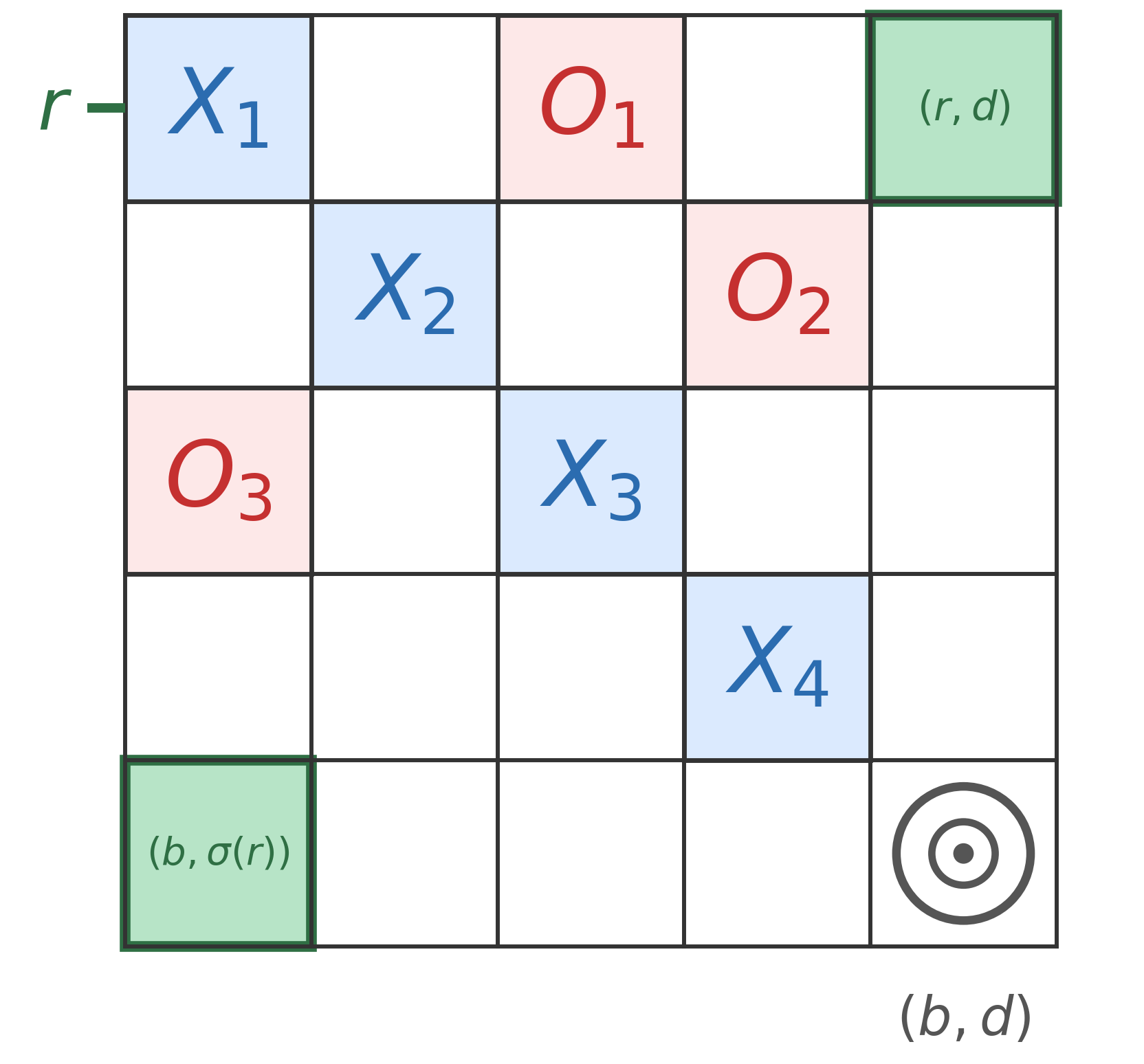}
    \caption{Row $r$ is live.}
    \label{fig:two-live-rows-r}
\end{subfigure}
\hfill
\begin{subfigure}[t]{0.43\textwidth}
    \centering
    \includegraphics[width=\textwidth]{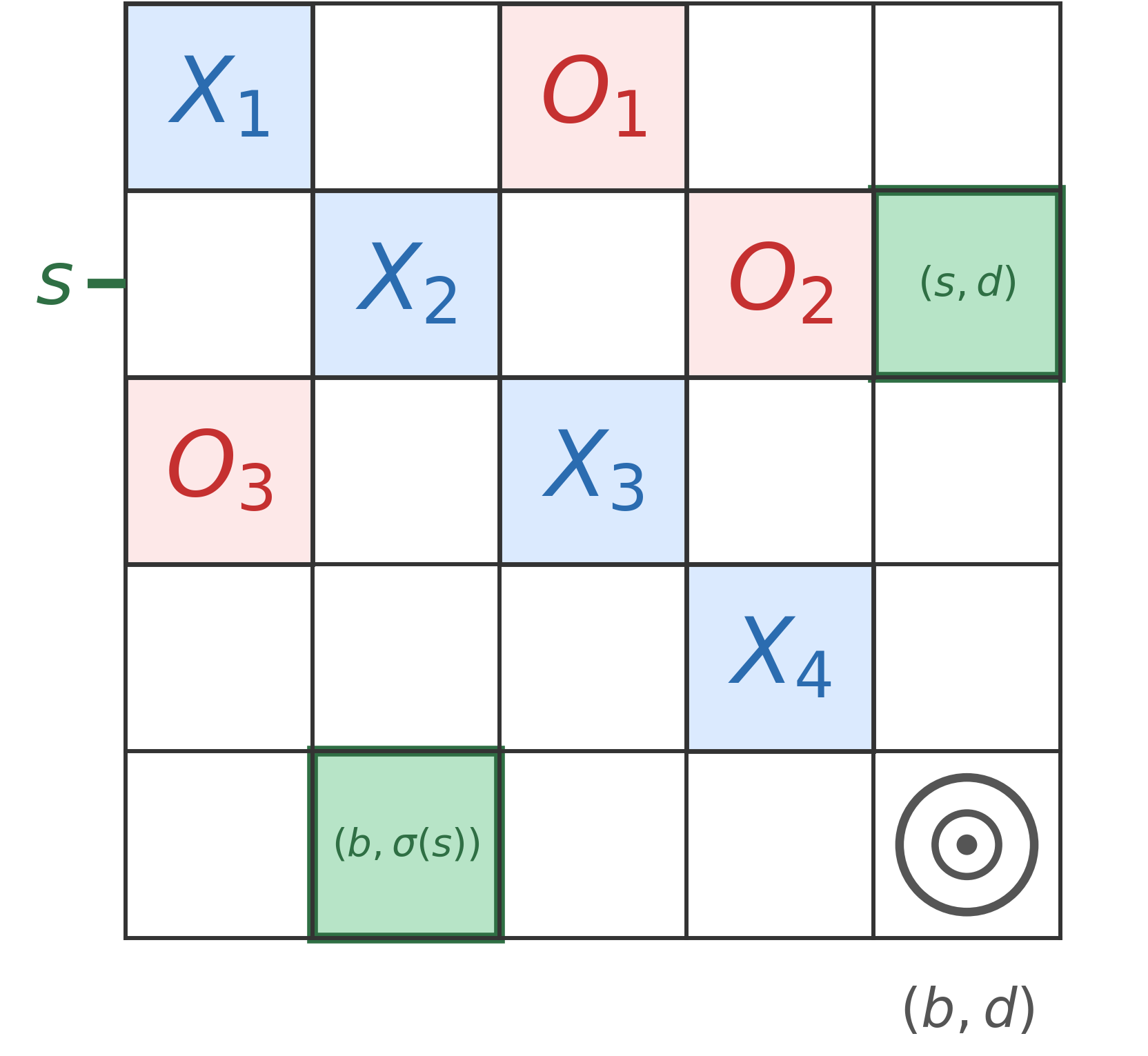}
    \caption{Row $s$ is live.}
    \label{fig:two-live-rows-s}
\end{subfigure}

\caption{
    Two live rows in the same Phase-2 position, with $(b, d)=(5, 5)$.
}
\label{fig:two-live-rows}
\end{figure}

\noindent\textbf{The two plans.} For distinct live rows $r,s$:

\begin{itemize}[leftmargin=2em]
\item \textbf{Plan (i)}---ply $2n-1$: play $(b,\sigma(r))$; ply $2n+1$: play $(s,d)$. Requires the \textit{cross cell} $(r,\sigma(s))$ to be free at ply $2n-3$ (Figure~\ref{fig:phase2-forcing}).
\item \textbf{Plan (ii)}---ply $2n-1$: play $(r,d)$; ply $2n+1$: play $(b,\sigma(s))$. Requires the \textit{cross cell} $(s,\sigma(r))$ to be free at ply $2n-3$.
\end{itemize}
\begin{figure}[ht]
\centering

\begin{subfigure}[t]{0.2829\textwidth}
    \centering
    \includegraphics[width=\textwidth]{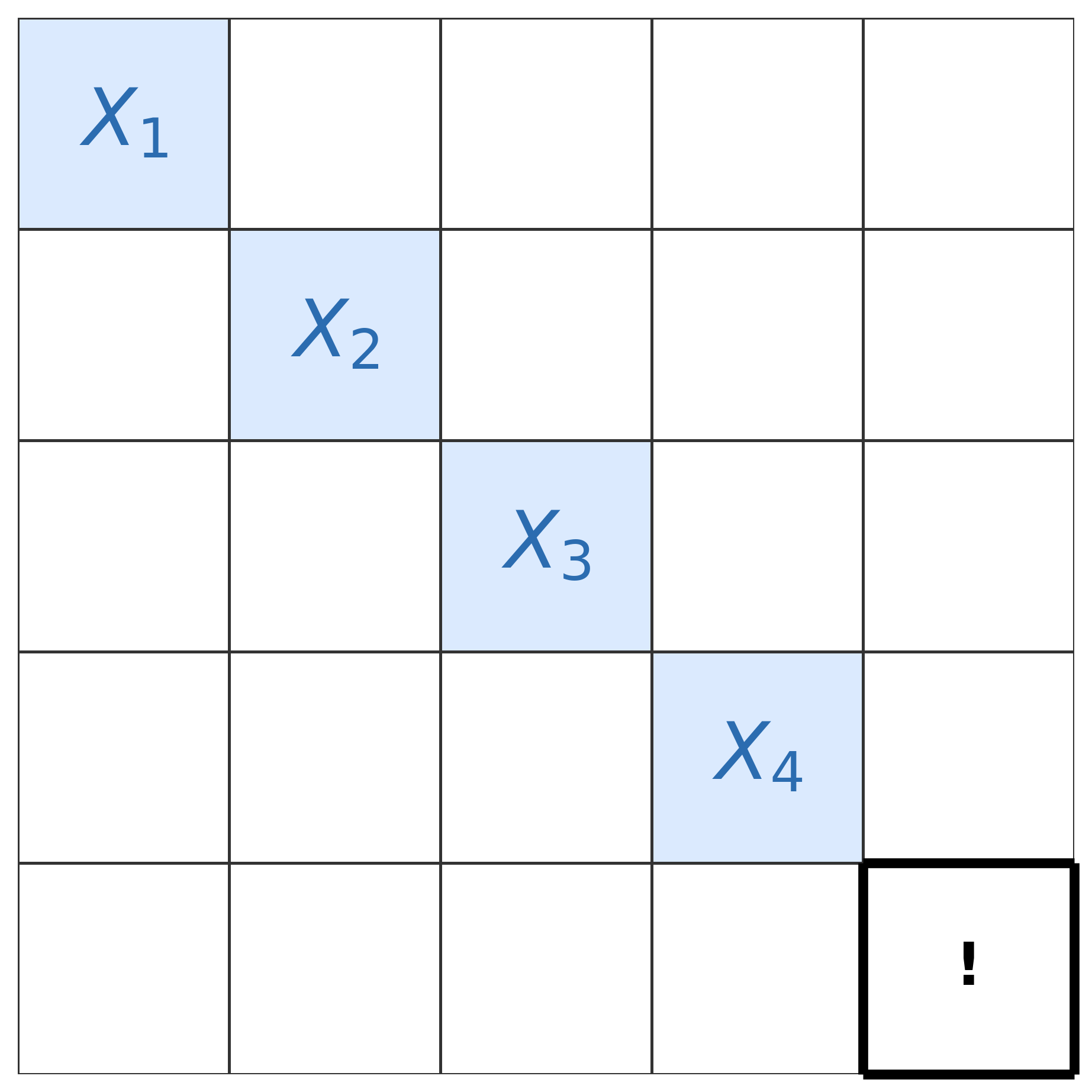}
    \caption{Initial $(n-1)$-matching.}
    \label{fig:phase2_matching}
\end{subfigure}
\hfill
\begin{subfigure}[t]{0.2829\textwidth}
    \centering
    \includegraphics[width=\textwidth]{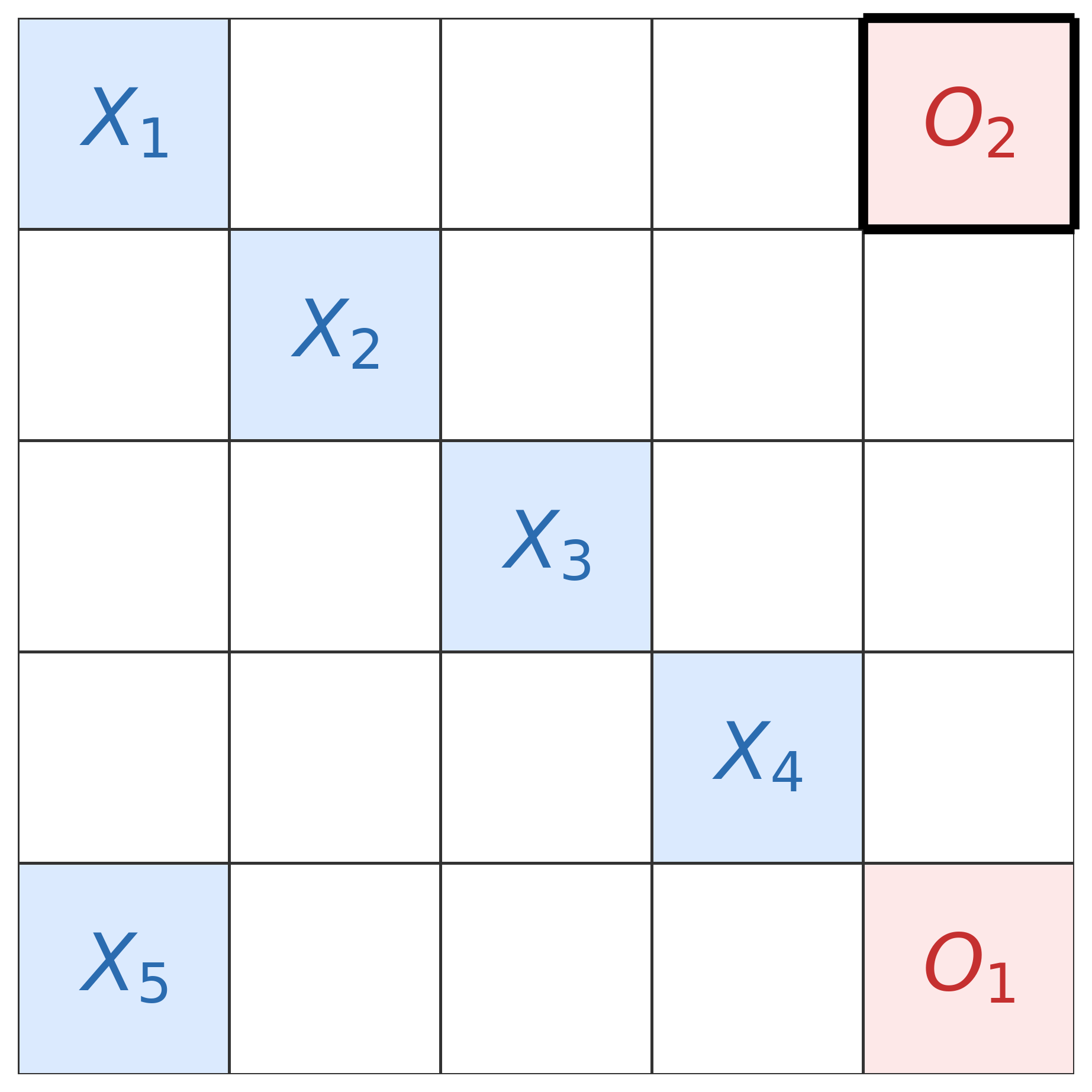}
    \caption{O blocks next threat.}
    \label{fig:phase2-single}
\end{subfigure}
\hfill
\begin{subfigure}[t]{0.2829\textwidth}
    \centering
    \includegraphics[width=\textwidth]{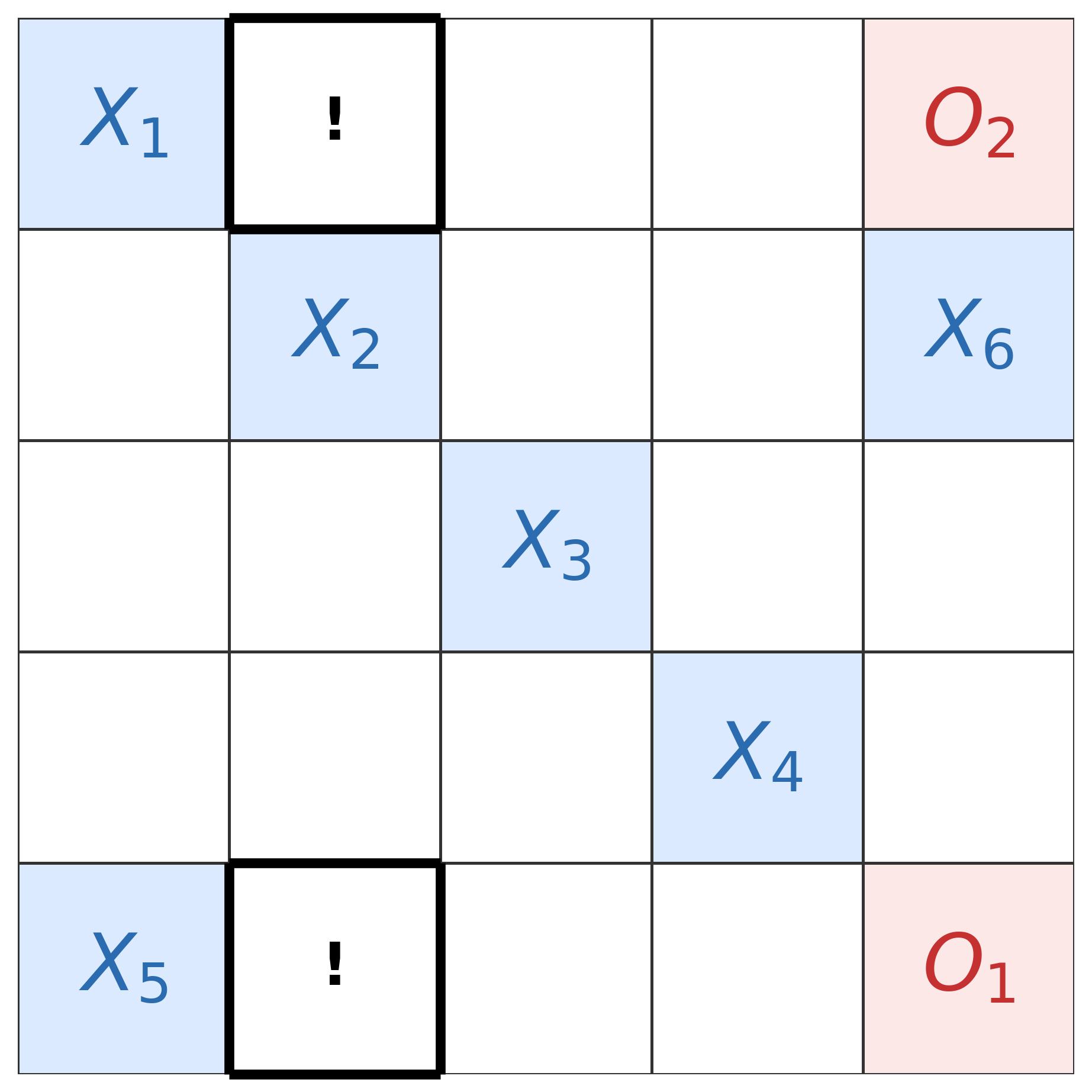}
    \caption{A double threat.}
    \label{fig:phase2-double}
\end{subfigure}

\caption{
    Phase 2 forcing mechanism. Here $n=5$, $\sigma(i)=i$, and the initial
    $(n-1)$-matching is
    $M=\{(1,1),(2,2),(3,3),(4,4)\}$.
    The unmatched row and column are $b=d=5$.
    We take $r=1$ and $s=2$.
    After $X_5$ is placed on $(b,\sigma(r))=(5,1)$,
    $O_2$ blocks the threat on $(r,d)=(1,5)$.
    $X_6$ on $(s,d)=(2,5)$ then threatens both $(5,2)$ and $(1,2)$.
}
\label{fig:phase2-forcing}
\end{figure}
\subsection*{X's move rule}

\begin{enumerate}[leftmargin=2em, label=\thesection.\arabic*, ref=\thesection.\arabic*]
\item \textbf{Moves $1, \dots, n-2$:} the Phase-1 rule above.
\item \textbf{Move $n-1$:} the tie-break above.
\item \label{rule:plan} \textbf{Plan and pair:}
  \begin{itemize}[leftmargin=1.5em]
  \item If $F$ meets column $d$: \textbf{plan (i)}, with any admissible pair $(r,s)$ (Lemma~\ref{lem:pair}).
  \item Else if $F$ meets row $b$: \textbf{plan (ii)}, with any admissible pair $(r,s)$ (Lemma~\ref{lem:pair}).
  \item Else ($w=0$): if $\nu(F)\le n-3$, \textbf{plan (i)} with any admissible pair; if $\nu(F)=n-2$, \textbf{plan (i)} with $s:=u_a$ (X's last Phase-1 row, in the notation of the tie-break above) and any live $r\ne u_a$---the cross cell is then automatically free (Lemma~\ref{lem:pairw0}).
  \end{itemize}
\item \label{rule:deviation} \textbf{If O ever declines to block}, X immediately plays a free completing cell of their own set and wins no later than ply $2n+3$. That such a cell exists is exactly what ``declines to block'' means: at plies $2n-2$ and $2n$, X has a unique free completing cell, and O's non-blocking move does not take it; at ply $2n+2$ they have two, and one move cannot take both.
\end{enumerate}

\noindent\textbf{Plan (i) against a blocking O.} At ply $2n-3$, X threatens $(b,d)$, and O blocks it. X plays $(b,\sigma(r))$: by Lemma~\ref{lem:rectangle}(a), the completing cells are $(b,d)$, which is occupied by O, and $(r,d)$, creating exactly one threat, so O must take $(r,d)$. X plays $(s,d)$: by Lemma~\ref{lem:rectangle}(c), the completing cells are $\{b,r\}\times\{d,\sigma(s)\}$, of which $(b,d),(r,d)$ are O's and $(b,\sigma(s)),(r,\sigma(s))$ are free, creating a double threat. O blocks one; X takes the other at ply $2n+3$ and wins. Plan (ii) is the mirror image via Lemma~\ref{lem:rectangle}(b),(d): O's forced blocks are $(b,d)$ then $(b,\sigma(r))$, and the final double threat is on $(s,d)$ and $(s,\sigma(r))$.

%---------------------------------------------------------------
\section{Proof of the Main Theorem}
\label{sec:proof}
%---------------------------------------------------------------

We now show that the strategy of \S\ref{sec:strategy} is well defined at every ply and that it wins by ply $2n+3$. The lemmas below supply, in order, the executability of Phase 1, the existence of a tie-break with $w\le n-3$, the existence of a suitable pair of live rows, the persistence of the cells X needs, and the impossibility of any useful deviation by O.

\medskip

\subsection*{Phase 1 is executable.}

\begin{lemma}[feasibility of Phase 1]
\label{lem:phase1}
The Phase-1 rule is always executable, and Invariant~\ref{inv:openblock} holds after each of X's moves $1,\dots,n-2$.
\end{lemma}

\begin{proof}
X's stones lie outside $U_R\times U_C$ by construction, so $H$ never contains an X-stone; $H$ only shrinks, so no stone can enter it. Inductively, before X's move $k$, the only O-stone that can lie in $H$ is O's most recent stone, $x_{k-1}$. Before move $k$, we have $|U_R|=|U_C|=m=n-k+1\ge 3$. If $x_{k-1}=(p,q)\in H$, the $m$ cells of row $p$ inside $H$ are free except for $x_{k-1}$ itself, leaving $m-1\ge2$ choices; playing one puts an X-stone in row $p$, deleting $p$ from $U_R$, so $x_{k-1}\notin H$ afterwards and the invariant is restored. If $x_{k-1}\notin H$, all $m^2$ cells of $H$ are free and preserve the invariant.
\end{proof}

\subsection*{The tie-break exists and gives $w\le n-3$.}

\begin{lemma}[good tie-break]
\label{lem:tiebreak}
Some admissible outcome has $w\le n-3$.
\end{lemma}

\begin{proof}
Assume every admissible outcome has $w > n-3$. Since $|F|=n-2$, this means every admissible outcome has $w=n-2$, i.e.\ $F\subseteq\text{row }b\,\cup\,\text{col }d$ in each case.

\emph{If $x_{n-2}=(u_i,v_j)\in H$:} intersecting the two constraints of the two admissible outcomes,
\[
(\text{row }u_i\cup\text{col }v_{3-j})\cap(\text{row }u_{3-i}\cup\text{col }v_j)=\{(u_i,v_j),(u_{3-i},v_{3-j})\}\subseteq H,
\]
using $\text{row }u_i\cap\text{row }u_{3-i}=\varnothing=\text{col }v_j\cap\text{col }v_{3-j}$. By Invariant~\ref{inv:openblock}, $H$ holds at most the one O-stone $x_{n-2}$, so $|F|\le1$, contradicting $|F|=n-2\ge2$ (Figure~\ref{fig:tiebreak-claims-b}).

\emph{If $x_{n-2}\notin H$:} intersect the constraints from the two outcomes $(b,d)=(u_1,v_1)$ and $(b,d)=(u_2,v_2)$ to get $F\subseteq\{(u_1,v_2),(u_2,v_1)\}\subseteq H$; but $H$ contains no O-stone at all, so $F=\varnothing$, again contradicting $n-2\ge2$.
\end{proof}

\begin{remark}
    This is where $n\ge4$ is first needed, as the argument requires $|F|\ge2$.
\end{remark}

\begin{lemma}[structure of $F$ when $w=0$]
\label{lem:structF}
Suppose the tie-break's outcome has $w=0$. Then $x_{n-2}\notin H$, all four corners are admissible, and $F\cap H=\varnothing$. Moreover, if $\nu(F)=n-2$, then $F$ is a perfect matching of $A\times B$ (Figure~\ref{fig:tiebreak-claims-c}).
\end{lemma}

\begin{proof}
Suppose $w = 0$. If $x_{n-2}\in H$, then both admissible outcomes have $w\ge1$, and by Lemma~\ref{lem:tiebreak}, one of them has $w\le n-3$; the tie-break would then have selected an outcome with $1\le w\le n-3$, contradicting $w=0$. So $x_{n-2}\notin H$, all four cells of $H$ are free and admissible, and by Invariant~\ref{inv:openblock}, no O-stone lies in $H$.

Now assume $\nu(F)=n-2$. Since $|F|=n-2$, $F$ is a matching of size $n-2$, and $w=0$ says $F$ misses row $b=u_{3-a}$ and column $d=v_{3-c}$. Suppose $F$ meets row $u_a$. All four corners are admissible, so consider the outcome $(b',d')=(u_a,v_{3-c})$. As $F$ is a matching, it meets row $u_a$ exactly once, and it misses column $v_{3-c}=d$; hence $w'=1+0=1\le n-3$ (using $n\ge4$), and the tie-break would have fired, contradicting $w=0$. Symmetrically, if $F$ meets column $v_c$, the admissible outcome $(b',d')=(u_{3-a},v_c)$ has $w'=0+1=1\le n-3$, yielding the same contradiction.

So $F$ misses rows $u_1,u_2$ and columns $v_1,v_2$, i.e.\ $F\subseteq A\times B$ with $|F|=n-2=|A|=|B|$; being a matching of that size, it covers every row of $A$ and every column of $B$.
\end{proof}

\begin{figure}[htbp]
\centering
\hfill
\begin{subfigure}[t]{0.5\textwidth}
\centering
\includegraphics[width=\textwidth]{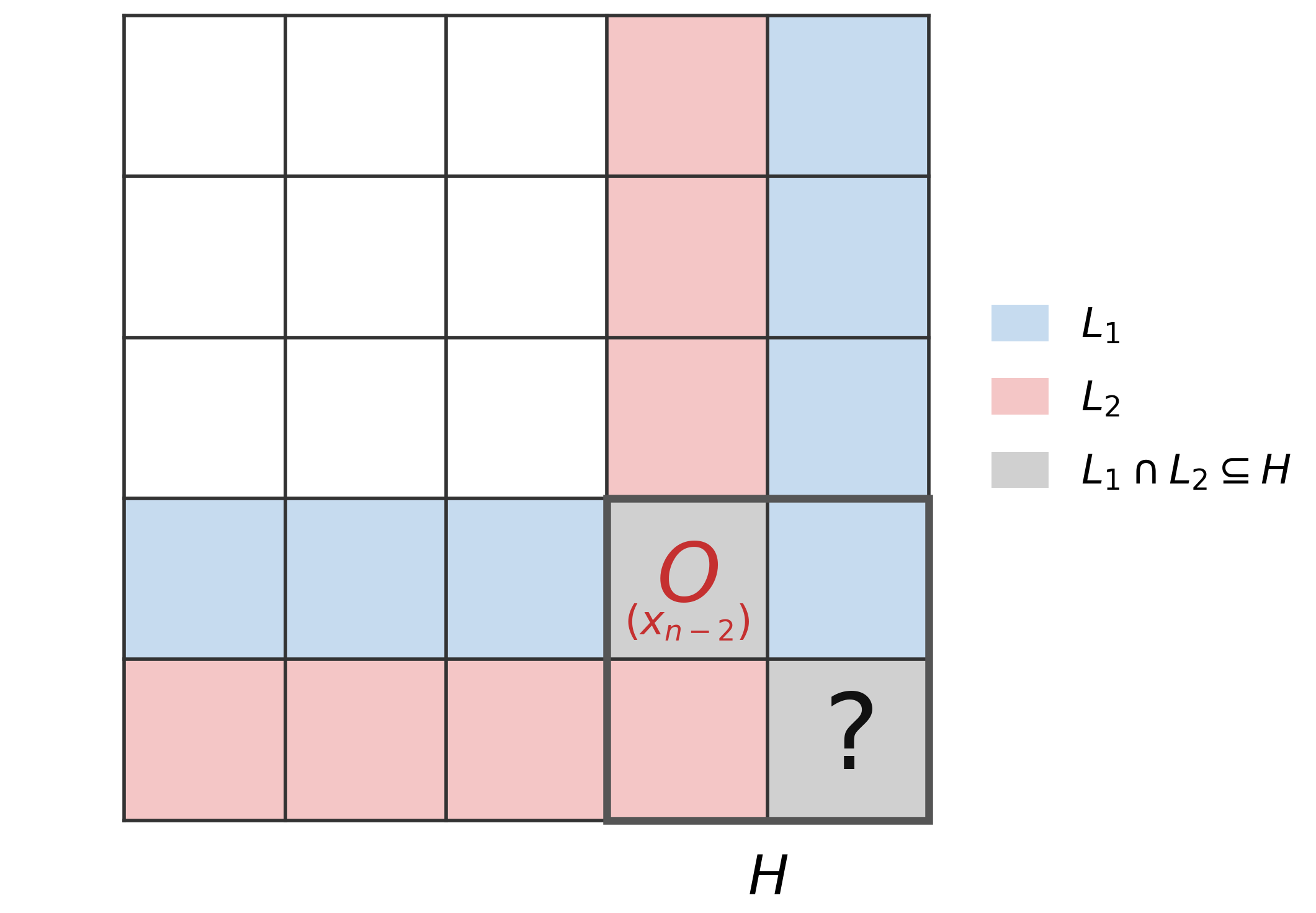}
\caption{Intersection argument used for Lemma~\ref{lem:tiebreak}.}
\label{fig:tiebreak-claims-b}
\end{subfigure}
\hfill
\begin{subfigure}[t]{0.47\textwidth}
\centering
\includegraphics[width=\textwidth]{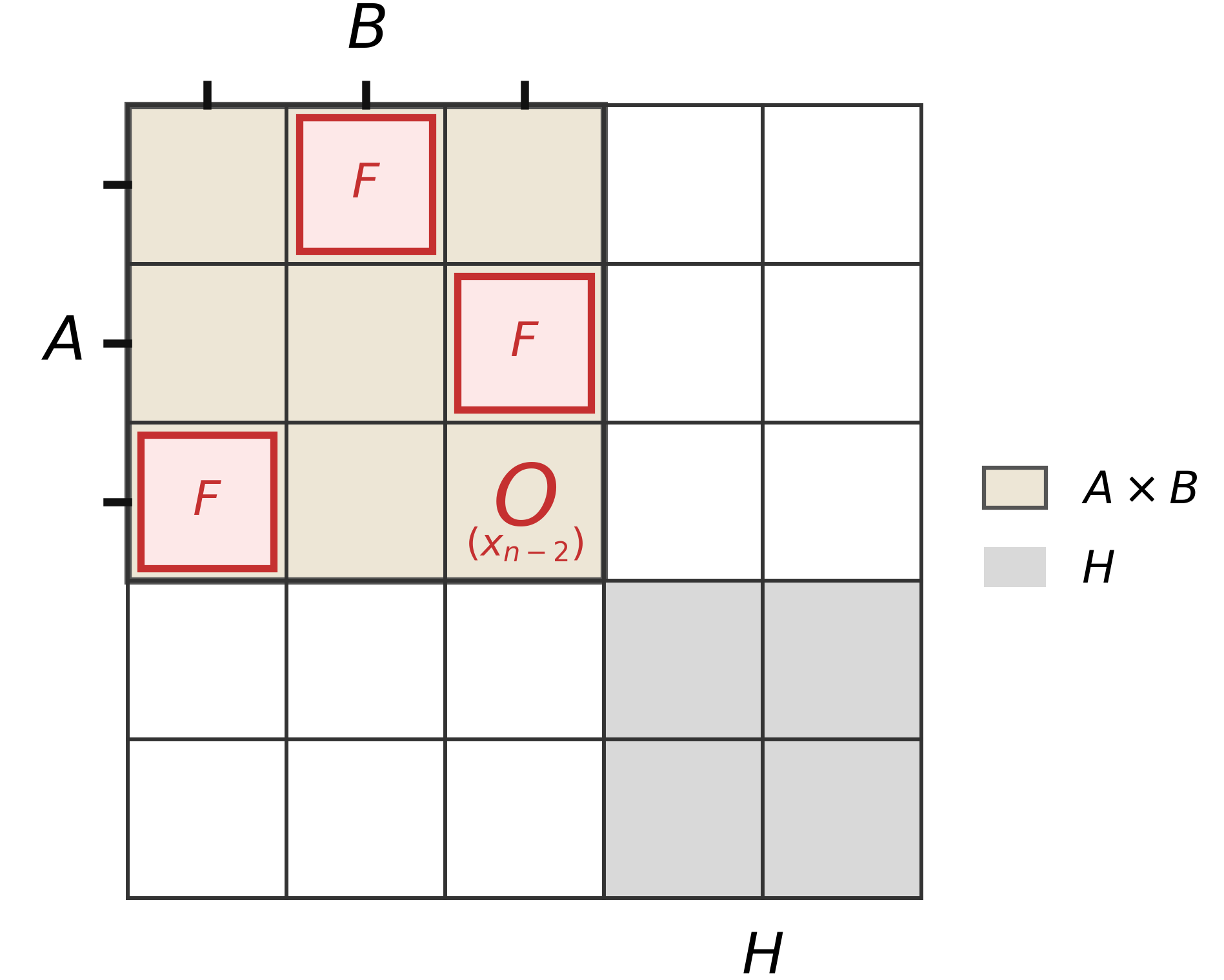}
\caption{Structure of $F$ in Lemma~\ref{lem:structF}.}
\label{fig:tiebreak-claims-c}
\end{subfigure}

\caption{
The structural arguments used in Lemmas~\ref{lem:tiebreak} and~\ref{lem:structF}. Panel~(a) shows the intersection of the line-pairs arising from the admissible outcomes in Lemma~\ref{lem:tiebreak}.
Panel~(b) shows the conclusion of Lemma~\ref{lem:structF} when $w=0$:
$F$ forms a perfect matching within
$A\times B$.
}
\label{fig:tiebreak-claims}
\end{figure}
\subsection*{Suitable live rows exist.}

\begin{lemma}[admissible pair]
\label{lem:pair}
If $w\le n-3$, X can choose $(r,s)$ for either plan.
\end{lemma}

\begin{proof}
By~\eqref{eq:live}, $\ell\ge n-1-w\ge2$, so ordered pairs of distinct live rows exist. For plan (i), the candidate cross cells are $(r,\sigma(s))$ with $r\ne s$ live; they lie outside row $b$ (as $r\ne b$) and outside column $d$ (as $\sigma(s)\ne d$), and distinct ordered pairs give distinct cells since $\sigma$ is injective. None of these is an X-stone: $(r,\sigma(s))\in M$ would force $\sigma(r)=\sigma(s)$, i.e.\ $r=s$. The only obstructions are thus O's stones outside $\text{row }b\cup\text{col }d$, of which there are exactly $n-2-w\ \ (\ge1)$, each killing at most one ordered pair. Since
\[
\ell(\ell-1)\ \ge\ (n-1-w)(n-2-w)\ \ge\ 2(n-2-w)\ >\ n-2-w,
\]
a valid pair survives. The same count applies to plan (ii), whose cross cells are $(s,\sigma(r))$.
\end{proof}

\begin{lemma}[admissible pair when $w=0$ and $\nu(F)=n-2$]
\label{lem:pairw0}
In case $w=0$ with $\nu(F)=n-2$, the pair $(r, s)$ given by $s=u_a$ and $r$ any live row $\ne u_a$ is valid for plan (i).
\end{lemma}

\begin{proof}
Here $w=0$, so $\ell=n-1$ and every row $\ne b$ is live, in particular $u_a$ (note $u_a\ne b=u_{3-a}$). Since $n-2\ge1$, some live $r\ne u_a$ exists. By Lemma~\ref{lem:structF}, $F$ is a perfect matching of $A\times B$. The cross cell is $(r,\sigma(u_a))=(r,v_c)$, and $r\in A$ because $r\notin\{u_a,u_{3-a}\}$. It is not owned by O, since $F\subseteq A\times B$ and $v_c\notin B$; it is not owned by X, since X's stones are $M_0\subseteq A\times B$ together with $(u_a,v_c)\ne(r,v_c)$. Therefore, it is free.
\end{proof}

\subsection*{Every cell that the plan needs stays free.}

\begin{lemma}[all cells X needs stay free]
\label{lem:free}
Assume $r\ne s$ are live and the relevant cross cell is free at ply $2n-3$,
and that O blocks at every forced ply. Then, in plan~(i), the cell
$(b,\sigma(r))$ is free at ply $2n-1$, the cell $(s,d)$ is free at ply
$2n+1$, and both $(b,\sigma(s))$ and $(r,\sigma(s))$ are free at ply
$2n+2$. The same holds for plan~(ii) under the mirror correspondence.
\end{lemma}
\begin{proof}
Under the hypothesis, the cells occupied between plies $2n-3$ and $2n+3$
are exactly $(b,d)$ and $(r,d)$ by O and $(b,\sigma(r))$ and $(s,d)$ by X,
together with O's final block in $\{(b,\sigma(s)),(r,\sigma(s))\}$. It
therefore suffices to check that each cell X requires differs from all
cells occupied before the ply at which X needs it.
    \begin{itemize}
    \item $(b,\sigma(r))$ is free at ply $2n-1$, as it differs from $(b,d)$ because $\sigma(r)\neq d$;
    \item $(s,d)$ is free at ply $2n+1$, as it differs from $(b,d),(r,d)$ because $s\ne b,r$, and from $(b,\sigma(r))$ because $s\neq b$;
    \item $(b,\sigma(s)),(r,\sigma(s))$ are free at ply $2n+2$, as they differ from $(b,d),(s,d)$ because $\sigma(s)\ne d$ and $r\ne b$, from $(r,d)$ likewise, and from $(b,\sigma(r))$ because $\sigma$ is injective and $s\ne r$.
\end{itemize}
The mirror computation gives plan (ii). If instead O deviates at some ply, their single move occupies at most one of X's free completing cells, so Rule~\ref{rule:deviation} still applies. 
\end{proof}

\subsection*{O cannot win or create a threat.} The win only requires the following. 
\begin{lemma}[no defensive deviation]
\label{lem:nodeviation}
Assume X follows \S\ref{sec:strategy} and O has blocked at every forced ply strictly before the ply in question. Then:
\begin{itemize}[leftmargin=2em]
\item[\textbf{(a)}] O has no threat at plies $2n-1$, $2n+1$ (when X must be free to execute their plan);
\item[\textbf{(b)}] Whatever cell O plays at ply $2n-2$, $2n$, or $2n+2$, they do not complete a transversal.
\end{itemize}
\end{lemma}

\begin{remark}[why a stronger statement is false]
Take $n=4$, X $=\{(1,1),(2,2),(3,3)\}$ and $F=\{(1,2),(2,1)\}$ at ply $5=2n-3$; here $b=d=4$, $\sigma=\mathrm{id}$, $u_a=v_c=3$, $w=0$, and $F$ is a perfect matching of $A\times B$ with $A=B=\{1,2\}$. Plan (i) with $s=u_a=3$, $r=1$ gives the line O(4,4), X(4,1), O(1,4), X(3,4), O(1,3), X(4,3) and X wins at ply $11=2n+3$. But after O's last block, their set $\{(1,2),(2,1),(4,4),(1,4),(1,3)\}$ has maximum matchings $\{(1,2),(2,1),(4,4)\}$ and $\{(1,3),(2,1),(4,4)\}$, so $D_R^O=\{3\}$, $D_C^O=\{2,3\}$ and O \emph{does} threaten the free cell $(3,2)$ at ply $2n+3$---harmlessly, as X completes first. Therefore, no claim about O's threats after ply $2n+2$ can or needs to be made.
\end{remark}

\begin{proof}[Proof of Lemma~\ref{lem:nodeviation}]
We use two facts. A matching meets each line at most once, so
\begin{equation}
\label{eq:line}
\nu(S)\ \le\ \nu(S\setminus L)+1\qquad\text{for every line }L;
\end{equation}
and, separately, adding a single cell raises $\nu$ by at most $1$, so
\begin{equation}
\label{eq:add}
\nu(S\cup T)\ \le\ \nu(S)+|T| .
\end{equation}

At ply $2n-2$, O holds $n-1<n$ stones, so (b) holds there by Corollary~\ref{cor:tempo}, whatever they play. By Lemma~\ref{lem:tiebreak}, the tie-break guarantees $w\le n-3$, so Lemma~\ref{lem:pair} applies. The three exhaustive cases below match Rule~\ref{rule:plan}.

\medskip

\noindent\textbf{(C-i) $F$ meets column $d$.} X uses plan (i), so O's blocks at plies $2n-2,2n$ are $(b,d),(r,d)$, both in column $d$. Hence when X is to move at plies $2n-1$ and $2n+1$, O's set is contained in $F\cup\text{col }d$, and by~\eqref{eq:line}
\[
\nu(O)\ \le\ \nu(F\setminus\text{col }d)+1\ \le\ (n-3)+1\ =\ n-2 ,
\]
using $|F\setminus\text{col }d|\le n-3$. By the last clause of Lemma~\ref{lem:threat}, O has no threat at plies $2n-1$ and $2n+1$: that is (a). For (b): immediately before O's moves at plies $2n$ and $2n+2$, their set satisfies $\nu(O)\le n-2$, so by~\eqref{eq:add}, any single addition leaves $\nu(O)\le n-1<n$. Therefore, no move of theirs at those plies wins.

\medskip

\noindent\textbf{(C-ii) $F$ meets row $b$.} X uses plan (ii), so the blocks at plies $2n-2,2n$ are $(b,d),(b,\sigma(r))$, both in row $b$; the identical computation with $L=\text{row }b$ gives $\nu(O)\le n-2$ before each of X's moves at plies $2n-1,2n+1$, and (a),(b) follow exactly as in (C-i).

\medskip

\noindent\textbf{(C-iii) $w=0$.} By Lemma~\ref{lem:structF}, $x_{n-2}\notin H$ and $F\cap H=\varnothing$.

\medskip

\noindent\emph{Sub-case $\nu(F)\le n-3$.} X uses plan (i). Before X's moves at plies $2n-1$ and $2n+1$, O's set is $F$ plus at most two cells of column $d$, so by~\eqref{eq:line}, $\nu(O)\le\nu(F)+1\le n-2$. As in (C-i), this gives (a), and (b) via~\eqref{eq:add}.

\medskip

\noindent\emph{Sub-case $\nu(F)=n-2$.} By Lemma~\ref{lem:structF}, $F$ is a perfect matching of $A\times B$. Here the crude bound~\eqref{eq:add} fails---O's set already has $\nu=n-1$---and we instead argue with Lemma~\ref{lem:threat}.

\medskip

\noindent\emph{Ply $2n-1$ (and O's move at ply $2n$).} O's set is $F\cup\{(b,d)\}$, a matching of size $n-1$ and nothing else, missing row $u_a$ and column $v_c$. By Corollary~\ref{cor:tempo}, its unique completing cell is $(u_a,v_c)$, which is occupied by X's last Phase-1 stone. Thus O has no threat at ply $2n-1$, giving (a); since no free cell completes their set, no move at ply $2n$ can win, giving (b).

\medskip

\noindent\emph{Ply $2n+1$ (and O's move at ply $2n+2$).} X plays plan (i) with $s=u_a$ and live $r\ne u_a$ (Lemma~\ref{lem:pairw0}), so O's set is $F\cup\{(b,d),(r,d)\}$. Row $u_a$ has no O-stones, so $\nu(O)\le n-1$; the only maximum matching is $F\cup\{(b,d)\}$, because $(b,d)$ and $(r,d)$ share column $d$, and replacing $(b,d)$ by $(r,d)$ forces $F$'s stone in row $r$ out, losing a cell. Therefore, $D_R^O=\{u_a\}$, $D_C^O=\{v_c\}$ and the unique completing cell $(u_a,v_c)$ is again occupied: there is no threat at ply $2n+1$, and no move at ply $2n+2$ wins.
\end{proof}

\subsection*{X wins by ply $2n+3$.}
\begin{proof}[Proof of Theorem A for $n\ge4$]
By Lemma~\ref{lem:phase1}, Phase 1 is executable, and X is never obliged to respond during it. At ply $2n-3$, after the tie-break, X owns exactly an $(n-1)$-matching $M$ missing row $b$ and column $d$ with $(b,d)$ free, while O holds $n-2$ stones. By Corollary~\ref{cor:tempo}, X threatens $(b,d)$, and O has neither a threat nor a win, so O must play $(b,d)$ or lose at ply $2n-1$.

Lemma~\ref{lem:tiebreak} makes $w\le n-3$ available; Lemmas~\ref{lem:pair} and~\ref{lem:pairw0} then supply, for the plan selected by Rule~\ref{rule:plan}, live rows $r\ne s$ with a free cross cell, and Lemma~\ref{lem:free} shows that every cell X's plan requires is still free when they need it. Lemma~\ref{lem:nodeviation}(b) shows that no O move at plies $2n-2,2n,2n+2$ completes a transversal, and Lemma~\ref{lem:nodeviation}(a) shows that they have no threat when X moves at plies $2n-1$ and $2n+1$, so X is never obliged to answer. Consequently, O's only non-losing move at ply $2n-2$ is $(b,d)$, and at ply $2n$, the only non-losing move is the unique block named by the plan. At ply $2n+2$, O faces a double threat and has no non-losing move. Any deviation is disposed of immediately by Rule~\ref{rule:deviation}. By Lemma~\ref{lem:rectangle}(a),(c)---resp.\ (b),(d)---X's ply-$(2n+1)$ move creates a double threat on two free cells, and X completes a transversal at ply $2n+3$.
\end{proof}

\noindent\textbf{Where $n\ge4$ is needed.} The hypothesis is used in exactly two places:

\begin{enumerate}[leftmargin=2em]
\item \textbf{Lemma~\ref{lem:tiebreak}} needs $|F|=n-2\ge2$. At $n=3$, we have $|F|=1$ and the intersection argument yields no contradiction, so nothing forces $w\le n-3=0$.
\item \textbf{Lemma~\ref{lem:structF}} uses the step $w'=1\le n-3$, which needs $n\ge4$.
\end{enumerate}

By contrast, Lemma~\ref{lem:pair} survives at $n=3$ provided $w=0$ is available ($\ell(\ell-1)=2>1=n-2-w$), and Lemma~\ref{lem:pairw0} needs only $n-2\ge1$. Hence the counting fails at $n=3$ precisely at the tie-break, and by exactly one stone: O survives $n=3$ by one tempo.

%---------------------------------------------------------------
\section{Computational Verification}
\label{sec:verification}
%---------------------------------------------------------------

To corroborate the correctness of the strategy of \S\ref{sec:strategy}, we implemented it as stated, including every tie-break and pair-selection rule. The implementation performs an exhaustive game-tree search against every legal defense by O. Whenever O has several legal moves, all are
explored; whenever X has a choice allowed by the strategy (e.g.,
multiple admissible pairs $(r,s)$), one canonical choice is taken, as the
proof demonstrates that every admissible choice is winning. Table~\ref{tab:verification}
and the distributions of Appendix~\ref{app:distributions} report the
canonical-choice mode. Every terminal branch records the first ply on which X wins. In particular,
the search checks every possible deviation by O from the forced blocking line, independently verifying the role of Lemma~\ref{lem:nodeviation} in the proof.

\begin{table}[ht]
\centering
\begin{tabular}{rrrrr}
\toprule
$n$ & terminal lines & nodes explored & maximum win ply & runtime \\
\midrule
4 & 4\,875 & 6\,075 & 11 & 0.3 s\\
5 & 485\,760 & 550\,224 & 13 & 23.2 s\\
6 & 75\,799\,185 & 82\,103\,245 & 15 & 5\,725 s\\
\bottomrule
\end{tabular}
\caption{Summary of the exhaustive verification.}
\label{tab:verification}
\end{table}

For every $n=4,5,6$, every explored game ends with an X win by ply
$2n+3$, as predicted by Theorem~A.
Moreover, some branch reaches ply $2n+3$ in each case, so the upper bound is
attained for all three values. The search also records the proof case reached by each terminal line.
All three cases of \S\ref{sec:proof} occur in practice, including the exceptional sub-case
(C-iii) with $\nu(F)=n-2$, confirming that each branch of the analysis is necessary.

\medskip

The source code, machine-generated verification reports, Lean~4
formalization (Appendix~\ref{app:lean}), and scripts producing figures are available in a
public GitHub repository.\footnote{Repository:
\url{https://github.com/keverage-guan/transversal}.} Appendix~\ref{app:implementation} describes the representation, the matching computations, the search algorithm, and the list of invariants checked during the search. Appendix~\ref{app:distributions} details statistics from the lines explored during the search.

%---------------------------------------------------------------
\section{Discussion and Open Problems}
\label{sec:discussion}
%---------------------------------------------------------------

The present work does not address the broader extremal questions posed by
Ranđelović~\cite{randjelovic}, including the minimum size $f(n)$ of a
winning family of transversals, the threshold $g(n)$ above which every family
is winning, and related conjectures. A natural direction for future work is determining whether the bound achieved
by our construction is optimal. 
\begin{conjecture}
\label{conj:sharp}
For every $n\ge4$ and every strategy of X, O has a defense surviving to
ply $2n+3$; that is, $2n+3$ is the exact length of the game under optimal
play.
\end{conjecture}

Exhaustive search confirms that for $n=4,5,6$, the strategy of
\S\ref{sec:strategy} admits a legal O defense forcing the win to ply
$2n+3$ (Table~\ref{tab:winply}). This does not rule out a faster strategy
for X, and establishing the lower bound for general $n$ remains open.

\section*{Acknowledgements}
The author thanks Noga Alon for reviewing the paper and for his helpful feedback.

\bibliographystyle{plain}
\bibliography{transversal}
%---------------------------------------------------------------
\appendix

\section{Implementation Details}
\label{app:implementation}

To complement the proof, we performed an exhaustive computer verification of the
strategy described in Sections~\ref{sec:strategy}--\ref{sec:proof}. The verifier
checks the strategy against every legal defense by O and verifies the
structural claims used in the proof at every position encountered. The complete
source code, together with the scripts used to generate the computational
figures and the machine-generated verification reports, is available in the accompanying GitHub repository.

\subsection{Verification Procedure}

Positions are represented using row bitmasks for the X- and
O-stones. The verifier computes the maximum matching number $\nu(S)$
using an augmenting-path algorithm, with computed
values memoized across the search.

The search follows the strategy prescribed in \S\ref{sec:strategy} and
branches over every legal move by O. Whenever the proof permits
multiple admissible choices, the verifier can either select a fixed
representative or branch over all such choices. In particular, the latter mode
branches over every admissible tie-breaking move and every admissible pair
$(r,s)$. Since all first moves are equivalent under row and column
permutations, only one representative first move is considered.

Completing cells are computed using the characterization of
Lemma~\ref{lem:threat}: when $\nu(S)=n-1$, the verifier identifies the sets
$D_R$ and $D_C$ from a maximum matching and returns the completing rectangle
$D_R\times D_C$. An independent routine that tests each free cell directly is
also included as a consistency check.

At every position, the verifier checks the structural assertions required by
the proof. These include the Phase~1 invariant
$H=U_R\times U_C$
being free of O-stones, the existence and validity of the strategy
context after the tie-breaking move, the bound $w\le n-3$, the prescribed
numbers of completing cells after the planned moves, and the absence of an
O-winning move. Any failed assertion produces an explicit
counterexample. The search terminates only when an X-transversal has
been obtained.

Thus the computation verifies not only that the strategy wins, but also that
all intermediate invariants on which the proof relies hold throughout the
entire search.

\subsection{Reproducibility}

The verification has been run exhaustively for $n=4,5,6$. The accompanying
source code reproduces these computations via

\begin{verbatim}
python3 transversal_verify.py --n 4
python3 transversal_verify.py --n 5
python3 transversal_verify.py --n 6
\end{verbatim}

The verification reports record the number of explored game histories and
search nodes, the winning plies, and summary statistics for the strategy
cases. The complete reports are included in the accompanying GitHub repository.

%---------------------------------------------------------------
\section{Exhaustive Search Statistics}
\label{app:distributions}
%---------------------------------------------------------------

This appendix records the distributions extracted from the machine-generated
verification reports for $n=4,5,6$. All entries are proportions of the
terminal lines of the search (respectively $4\,875$, $485\,760$ and
$75\,799\,185$ lines); figures are
rounded to four decimals. A \textit{terminal line} is one complete play of the game
in which O follows an arbitrary legal defense and X follows the strategy of
\S\ref{sec:strategy}, with the canonical choice taken at every point where the
strategy admits several moves. Every terminal line ends in an X win, and only three winning
plies occur: $2n-1$, $2n+1$ and $2n+3$.

\begin{table}[ht]
\centering
\begin{tabular}{l|rrr}
\toprule
& \multicolumn{3}{c}{win ply} \\
\cmidrule(lr){2-4}
$n$ & $2n-1$ & $2n+1$ & $2n+3$ \\
\midrule
4 & 0.4000 & 0.3200 & 0.2800 \\
5 & 0.3696 & 0.3261 & 0.3043 \\
6 & 0.3562 & 0.3288 & 0.3151 \\
\bottomrule
\end{tabular}
\caption{Distribution of the winning ply, as a proportion of terminal
lines. In absolute terms, the columns are plies $7,9,11$ for $n=4$;
$9,11,13$ for $n=5$; and $11,13,15$ for $n=6$.}
\label{tab:winply}
\end{table}

\medskip

Table~\ref{tab:cases} gives the proportion of
terminal lines reaching each of the cases of the proof of
Lemma~\ref{lem:nodeviation}. The exceptional sub-case (C-iii) with
$\nu(F)=n-2$, in which the bound~\eqref{eq:add} fails and
Lemma~\ref{lem:threat} must be invoked instead, is rare and becomes rarer as
$n$ grows.

\begin{table}[ht]
\centering
\begin{tabular}{l|rrrr}
\toprule
& \multicolumn{4}{c}{case of Lemma~\ref{lem:nodeviation}} \\
\cmidrule(lr){2-5}
$n$ & (C-i) & (C-ii) & (C-iii), $\nu(F)\le n-3$ & (C-iii), $\nu(F)=n-2$ \\
\midrule
4 & 0.5538 & 0.3692 & 0.0667 & 0.0103 \\
5 & 0.7145 & 0.2562 & 0.0270 & 0.0023 \\
6 & 0.7498 & 0.2188 & 0.0309 & 0.0005 \\
\bottomrule
\end{tabular}
\caption{Distribution of the proof case reached, as a proportion of all
terminal lines.}
\label{tab:cases}
\end{table}

\medskip

Table~\ref{tab:planw} collects
the distribution of the plan chosen by Rule~\ref{rule:plan} and of the
parameter $w=|F\cap(\text{row }b\cup\text{col }d)|$ on which the case split
turns. Plan~(i) is used on the lines falling in cases (C-i) and
(C-iii), and plan~(ii) on those in case (C-ii); plan~(i) is the more common choice at every board size, and increasingly so as $n$ grows. Entries left blank denote values of
$w$ that are impossible at that board size, excluded by
Lemma~\ref{lem:tiebreak}. The two blocks are normalized separately.

\begin{table}[ht]
\centering
\begin{tabular}{l|rr|rrrr}
\toprule
& \multicolumn{2}{c}{plan} & \multicolumn{4}{c}{$w$} \\
\cmidrule(lr){2-3}\cmidrule(lr){4-7}
$n$ & (i) & (ii) & $0$ & $1$ & $2$ & $3$ \\
\midrule
4 & 0.6308 & 0.3692 & 0.0769 & 0.9231 & --- & --- \\
5 & 0.7438 & 0.2562 & 0.0293 & 0.6616 & 0.3091 & --- \\
6 & 0.7812 & 0.2188 & 0.0314 & 0.5938 & 0.3125 & 0.0623 \\
\bottomrule
\end{tabular}
\caption{Distribution of the plan selected by Rule~\ref{rule:plan} and of the
parameter $w$, each as a proportion of all terminal lines.}
\label{tab:planw}
\end{table}
\section{Formal Verification in Lean}
\label{app:lean}

As an independent check of the mathematical argument, we also formalized the main result in Lean~4. The formalization was generated with the assistance of Aristotle \cite{achim2025aristotle} and subsequently checked by the Lean kernel. The formalization encodes the transversal achievement game, the relevant notions of matchings and transversals, and the strategy used in the proof. The principal theorem corresponding to Theorem~A is stated and proved within the Lean development, along with the correctness of the $n=3$ draw and each of the lemmas used throughout the proof. The complete Lean source, together with the version information and instructions required to reproduce the verification, is included in the accompanying GitHub repository.

\section{Example Playthrough}
\label{app:playthrough}
The figures below show a complete playthrough of the strategy described in \S\ref{sec:strategy} on a $6 \times 6$ board. Each panel shows the position after the indicated ply, with the open block $H$ shaded gray. X wins on ply $15=2n+3$.

\begin{figure}[htbp]
\centering
\begin{subfigure}[t]{0.30\textwidth}
\centering
\includegraphics[width=\textwidth]{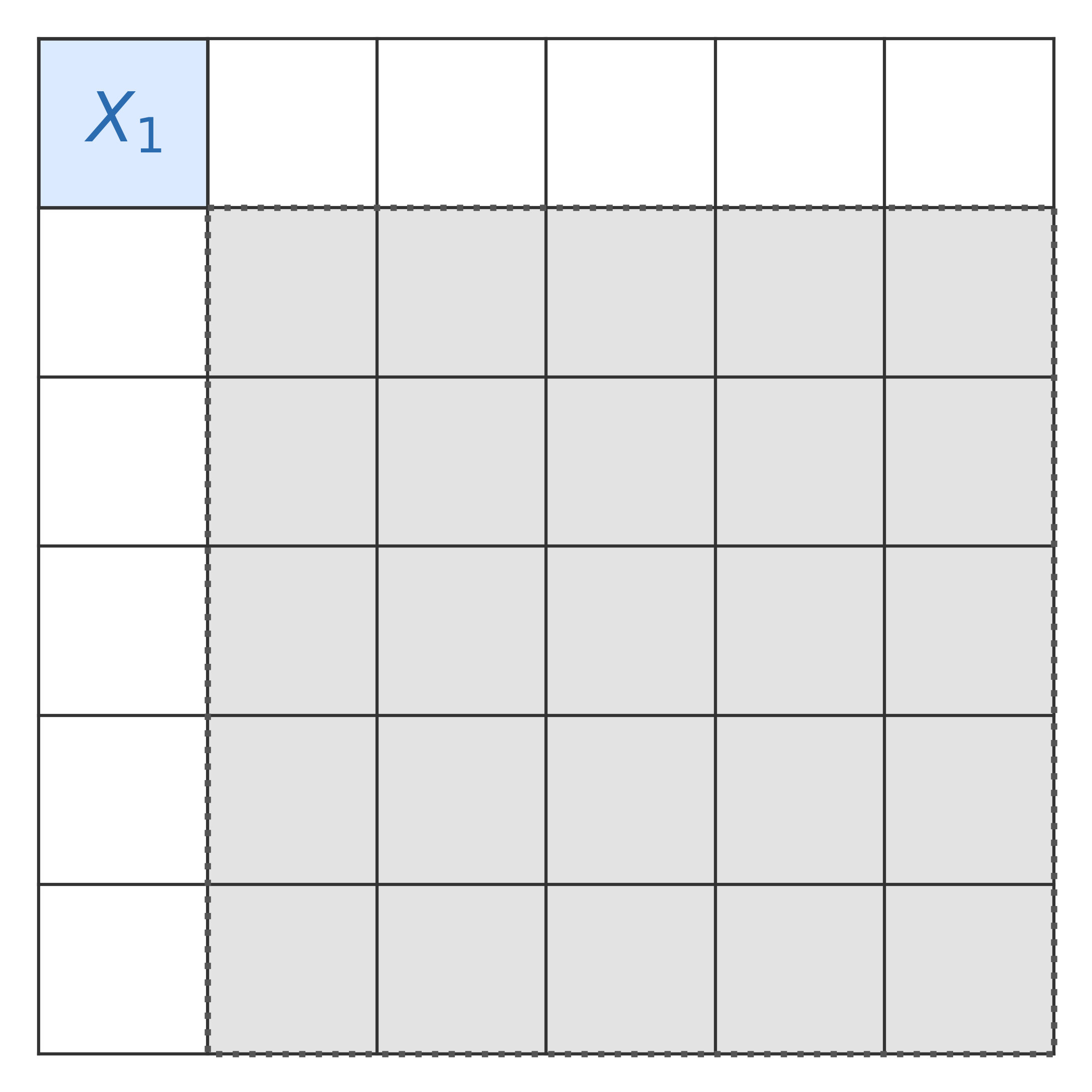}
\caption*{Ply 1: X plays arbitrary first move}
\end{subfigure}
\hfill
\begin{subfigure}[t]{0.30\textwidth}
\centering
\includegraphics[width=\textwidth]{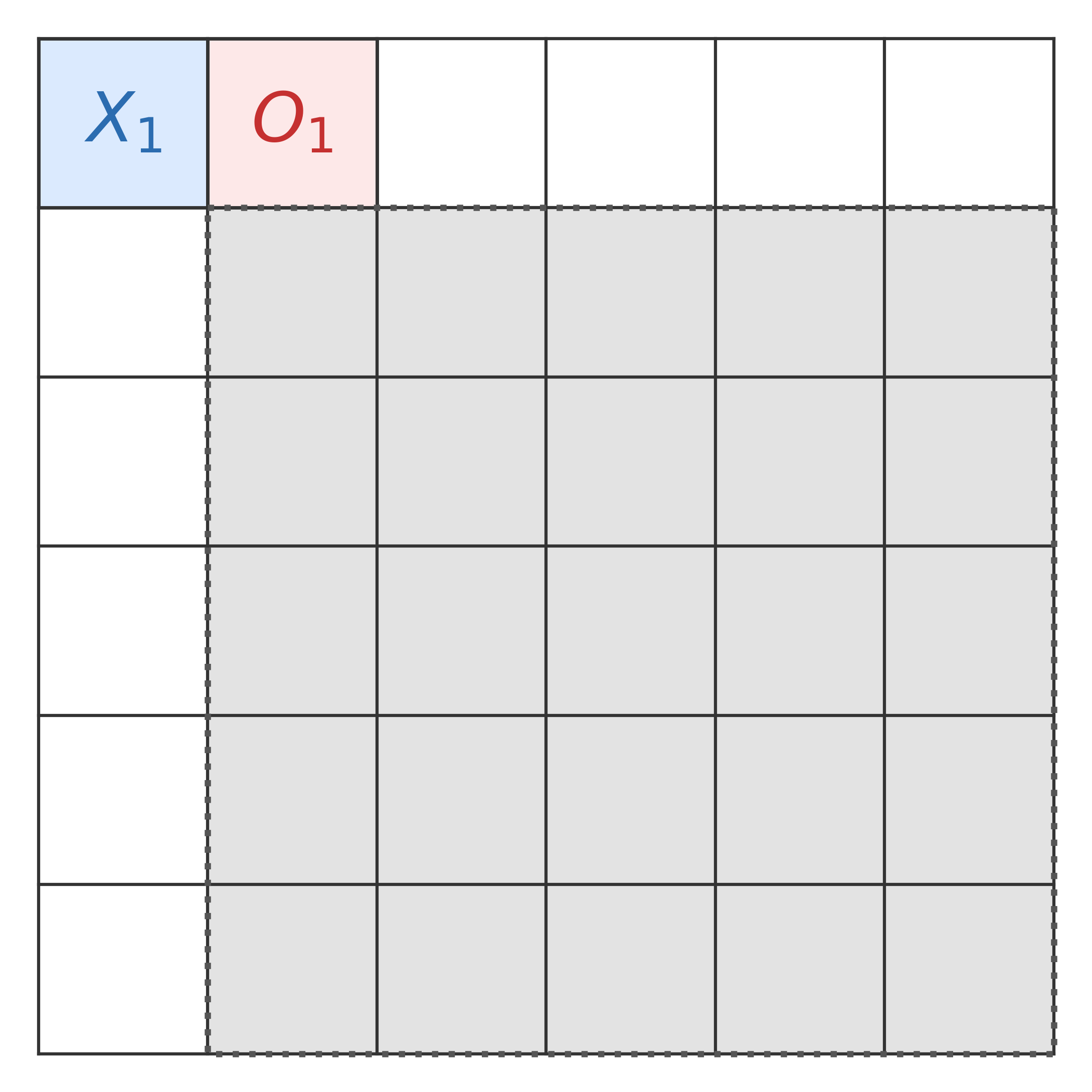}
\caption*{Ply 2: O plays outside $H$}
\end{subfigure}
\hfill
\begin{subfigure}[t]{0.30\textwidth}
\centering
\includegraphics[width=\textwidth]{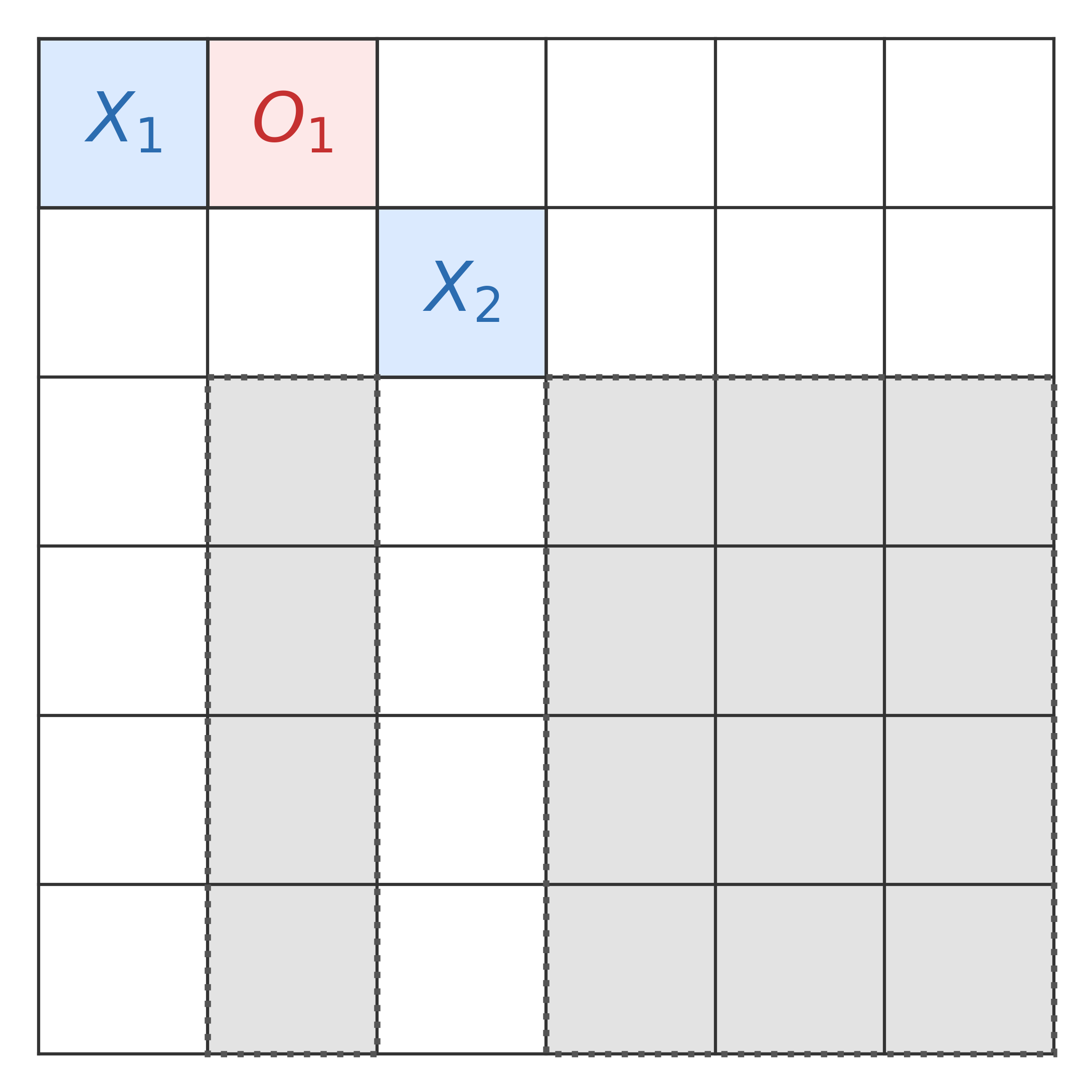}
\caption*{Ply 3: X plays arbitrary move inside $H$}
\end{subfigure}
\end{figure}

\begin{figure}[htbp]
\centering
\begin{subfigure}[t]{0.30\textwidth}
\centering
\includegraphics[width=\textwidth]{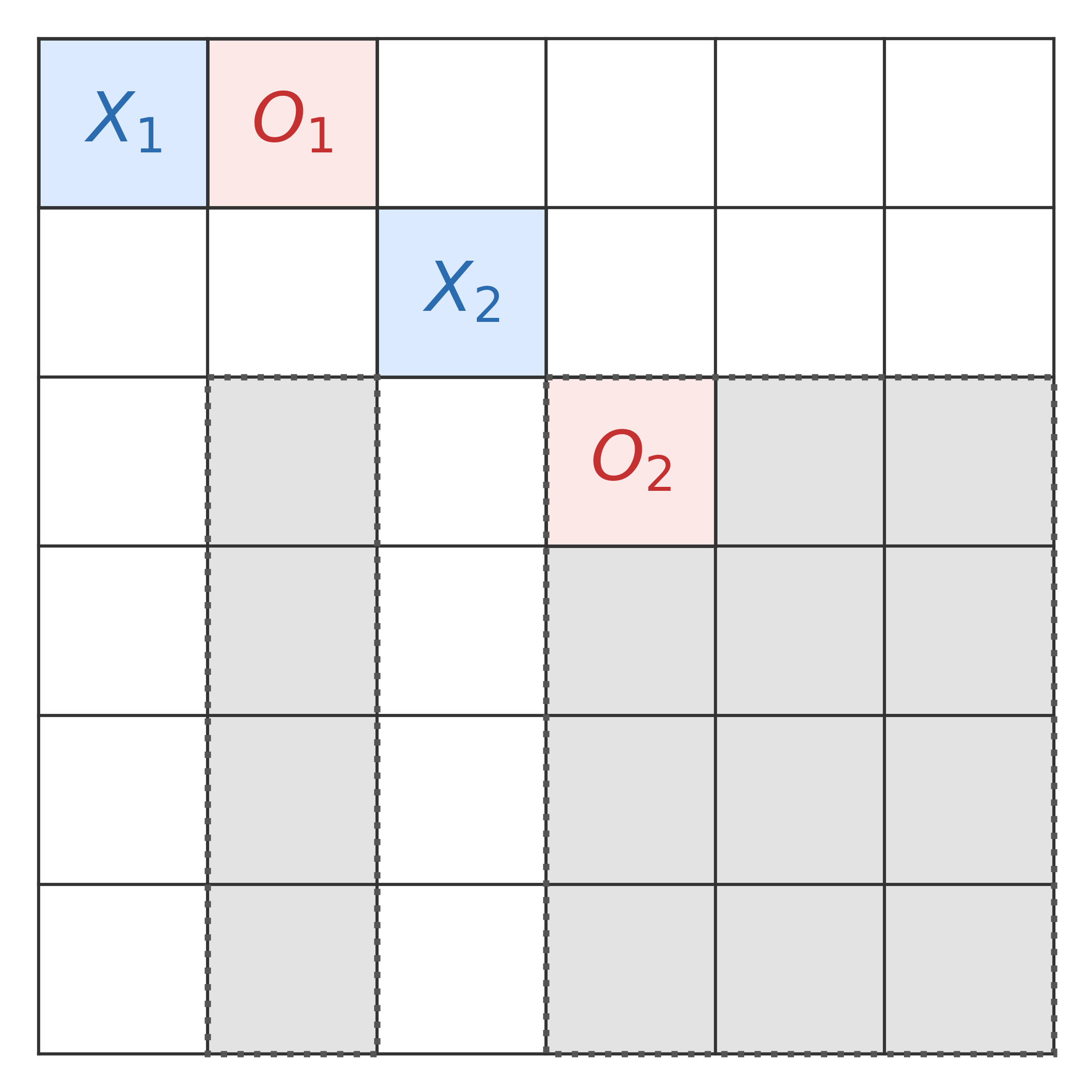}
\caption*{Ply 4: O plays inside $H$}
\end{subfigure}
\hfill
\begin{subfigure}[t]{0.30\textwidth}
\centering
\includegraphics[width=\textwidth]{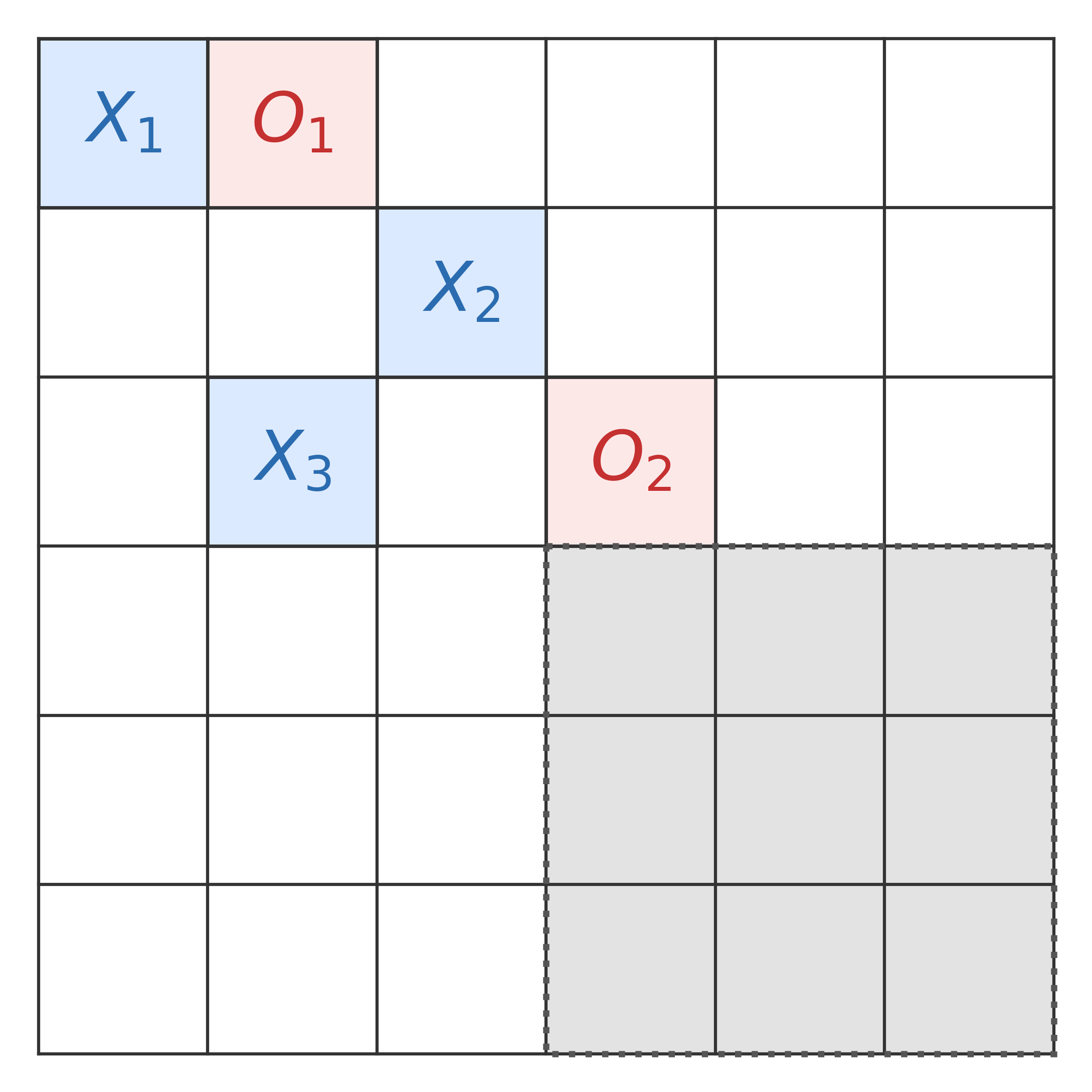}
\caption*{Ply 5: X plays in same row as O, restoring invariant}
\end{subfigure}
\hfill
\begin{subfigure}[t]{0.30\textwidth}
\centering
\includegraphics[width=\textwidth]{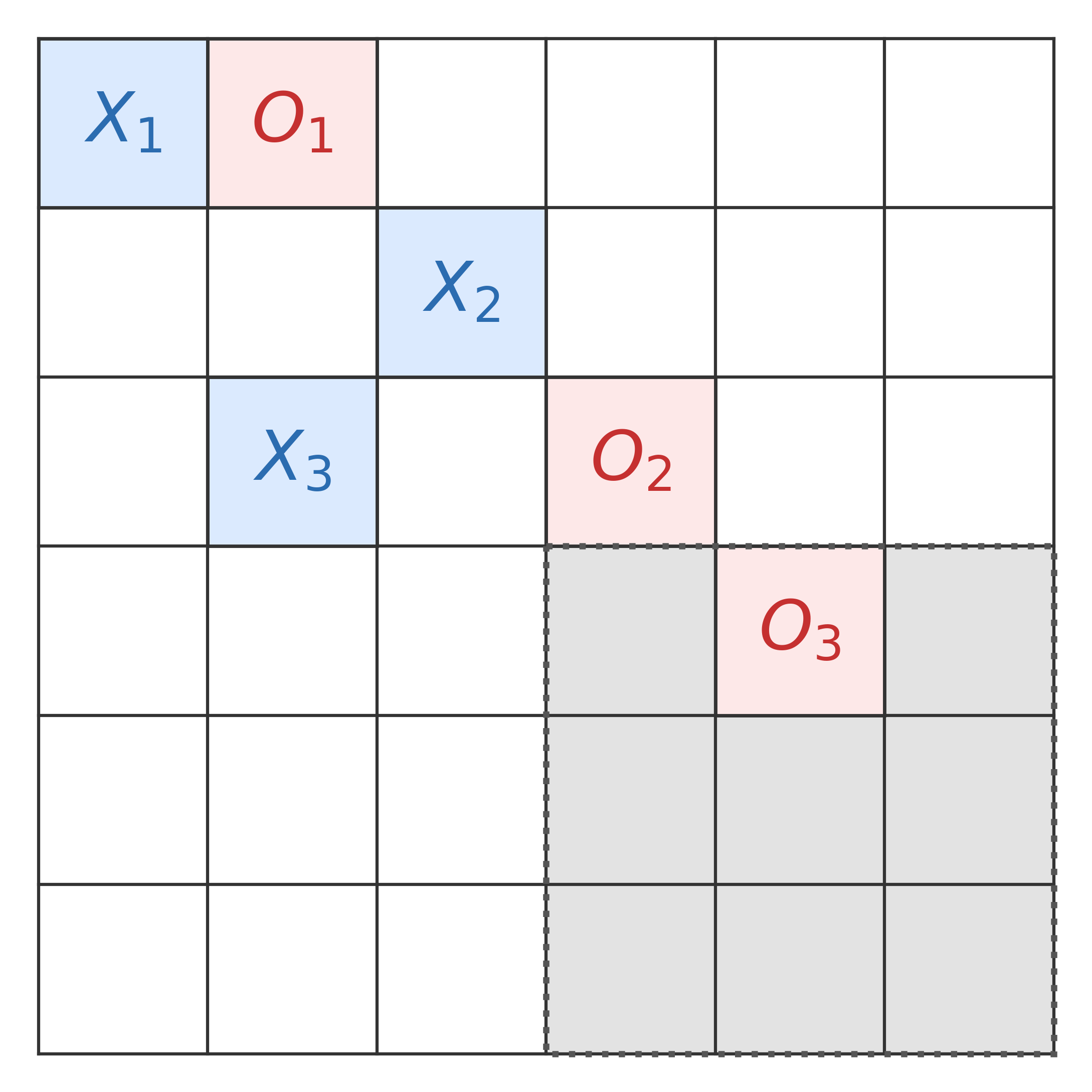}
\caption*{Ply 6: O plays inside $H$}
\end{subfigure}
\end{figure}

\begin{figure}[htbp]
\centering
\begin{subfigure}[t]{0.30\textwidth}
\centering
\includegraphics[width=\textwidth]{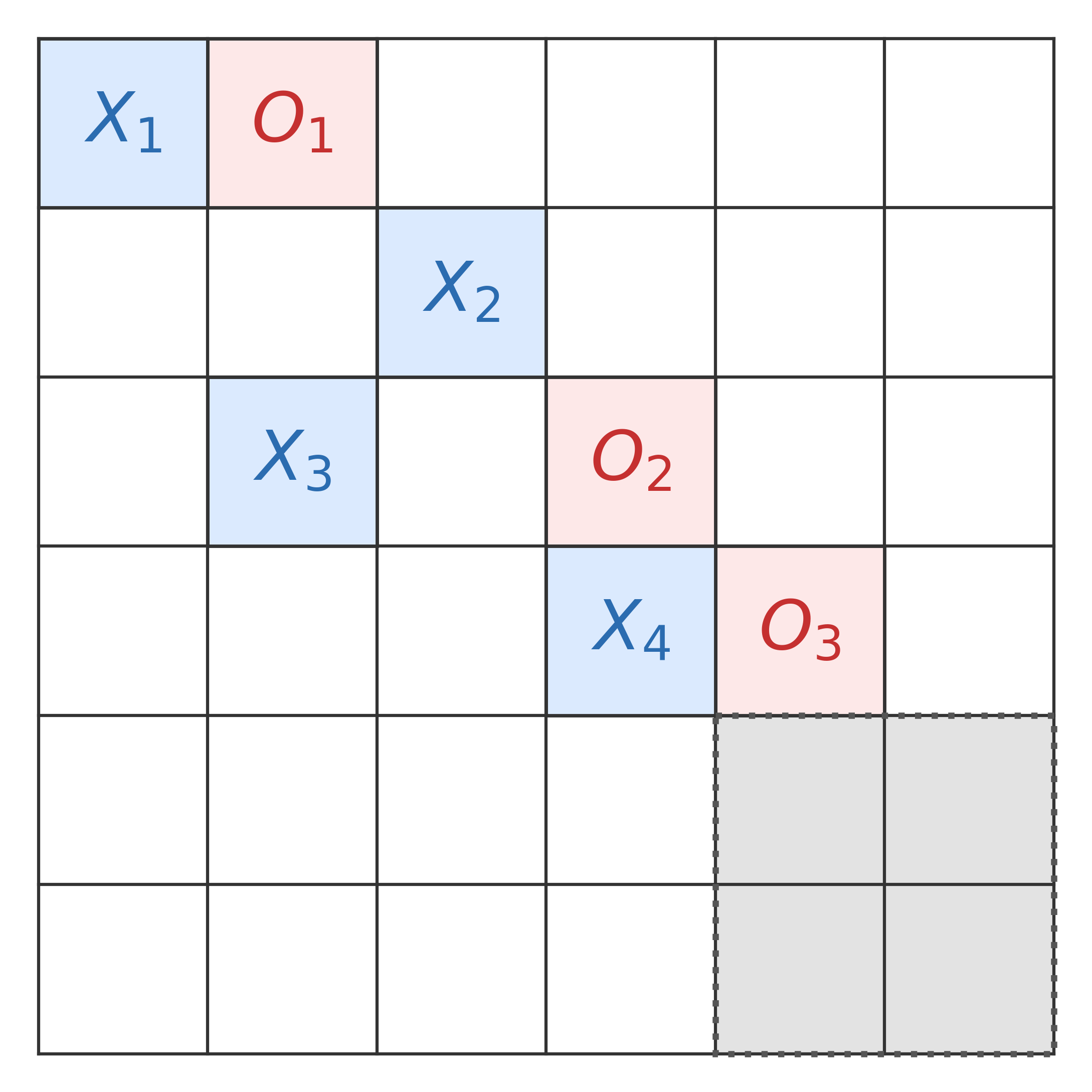}
\caption*{Ply 7: X plays in same row}
\end{subfigure}
\hfill
\begin{subfigure}[t]{0.30\textwidth}
\centering
\includegraphics[width=\textwidth]{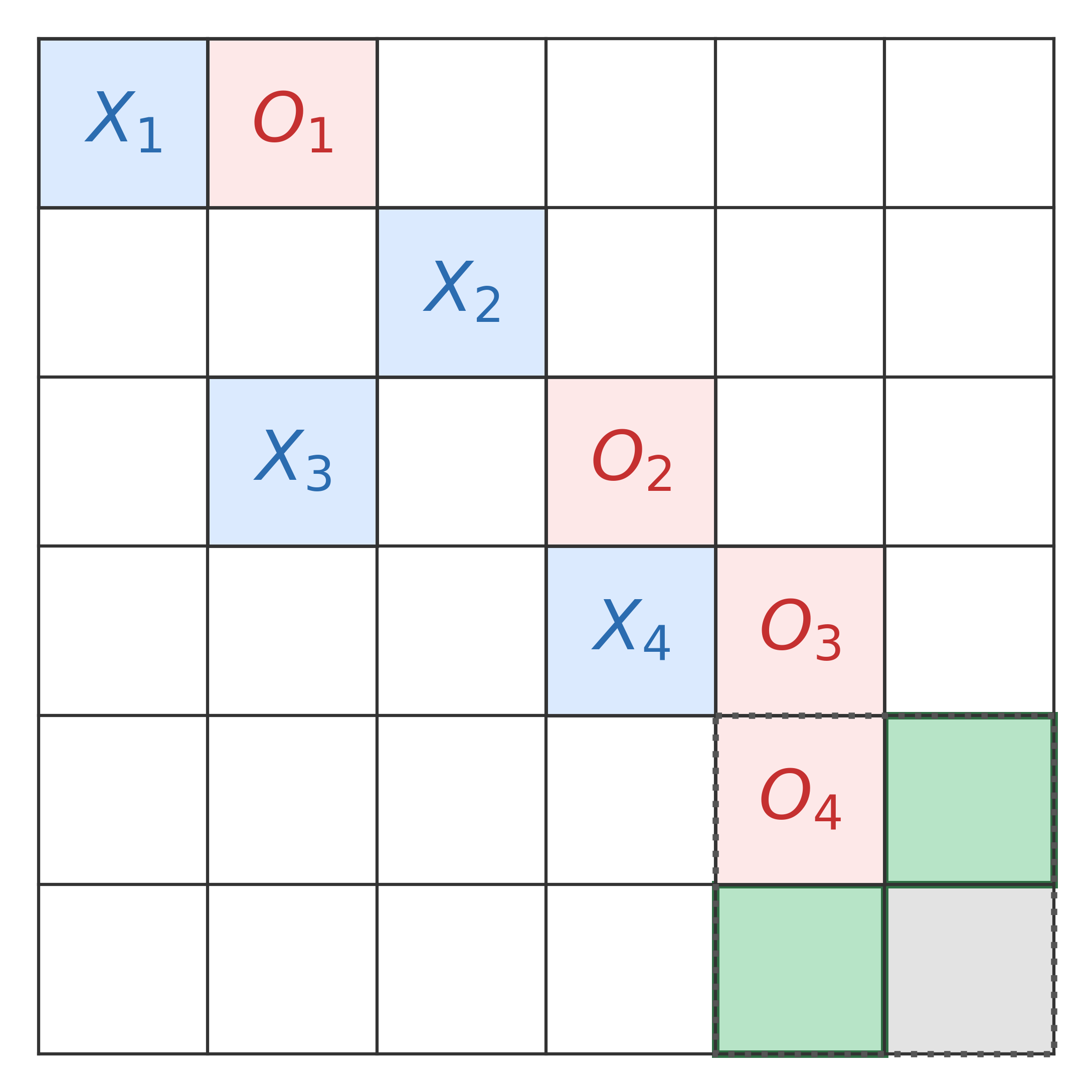}
\caption*{Ply 8: O plays inside $H$, leaving two admissible cells}
\end{subfigure}
\hfill
\begin{subfigure}[t]{0.30\textwidth}
\centering
\includegraphics[width=\textwidth]{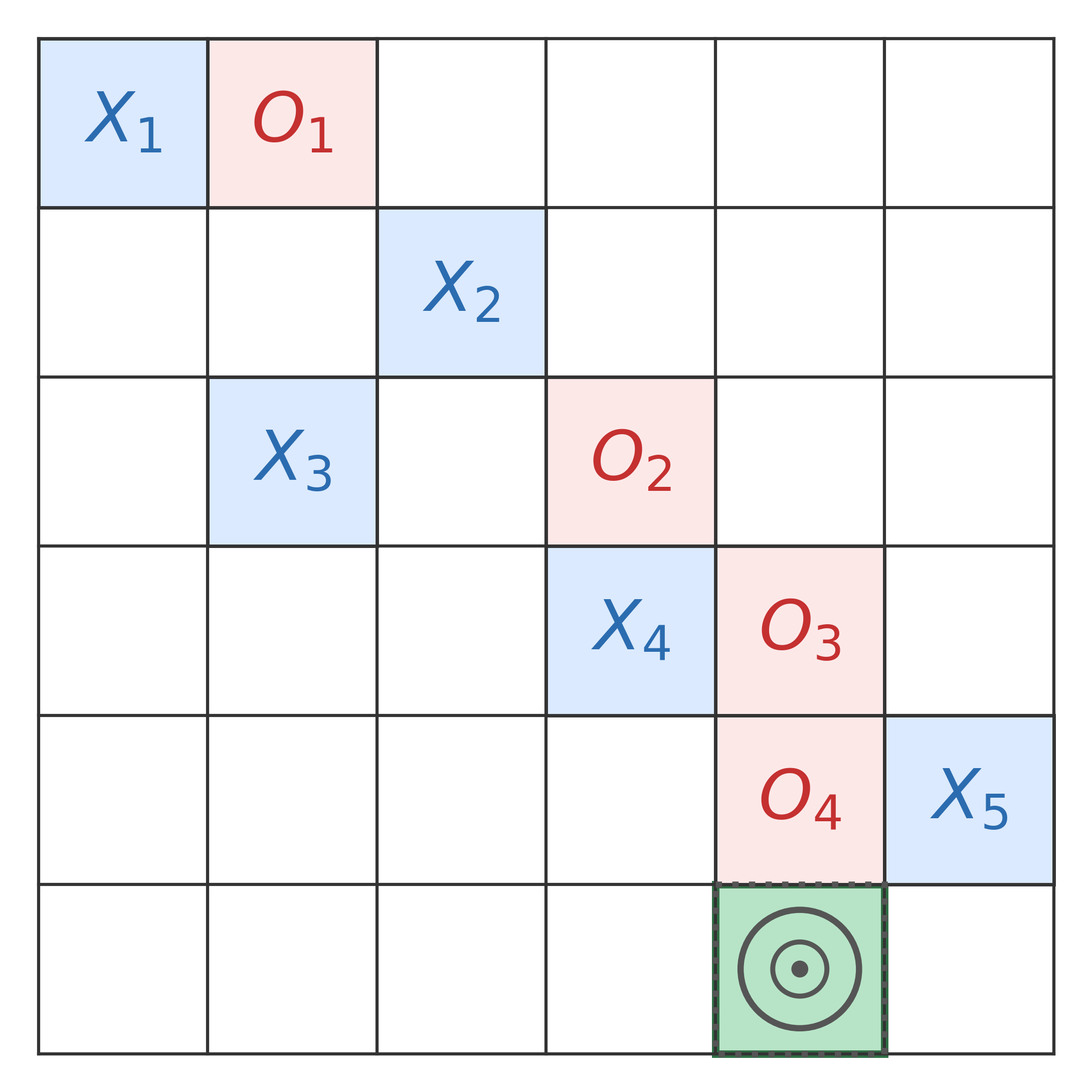}
\caption*{Ply 9: X plays admissible cell $(5,6)$ and threatens cell $(b,d)=(6,5)$; $F$ intersects column $d$}
\end{subfigure}
\end{figure}

\begin{figure}[htbp]
\centering
\begin{subfigure}[t]{0.30\textwidth}
\centering
\includegraphics[width=\textwidth]{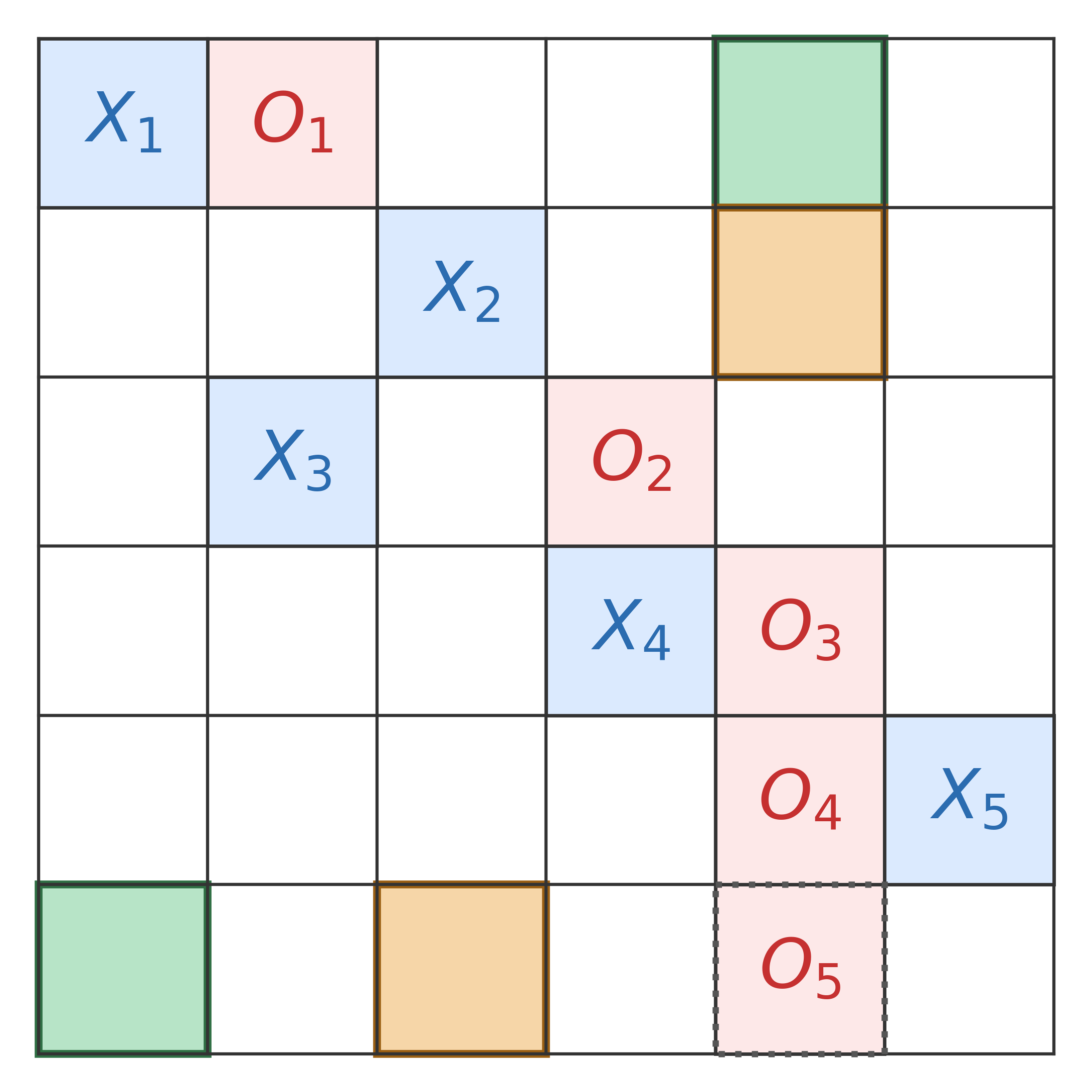}
\caption*{Ply 10: O blocks $(b,d)$; rows $r=1$, $s=2$ are live}
\end{subfigure}
\hfill
\begin{subfigure}[t]{0.30\textwidth}
\centering
\includegraphics[width=\textwidth]{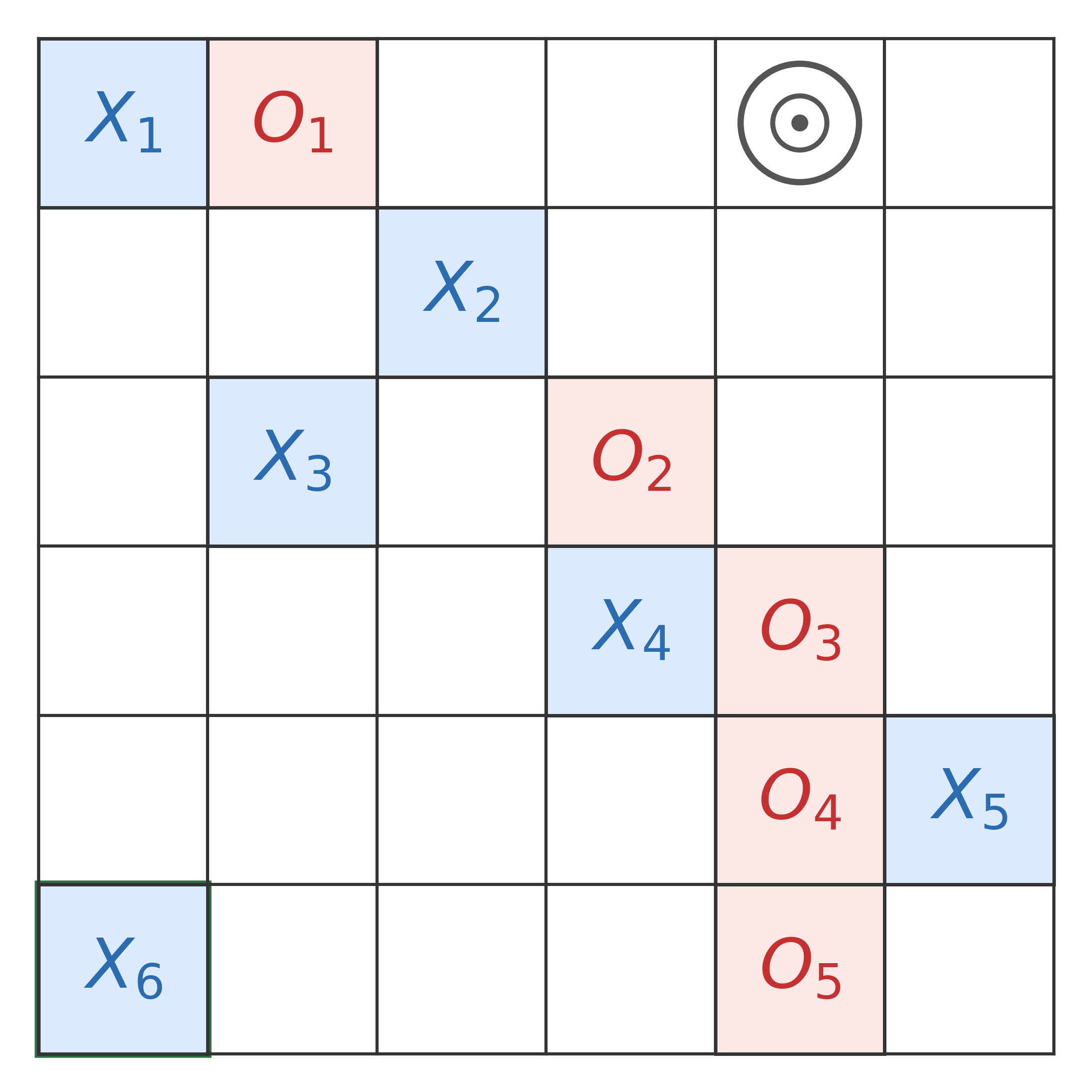}
\caption*{Ply 11: X follows plan (i) and plays $(b,\sigma(r))=(6,1)$, threatening $(r,d)=(1,5)$}
\end{subfigure}
\hfill
\begin{subfigure}[t]{0.30\textwidth}
\centering
\includegraphics[width=\textwidth]{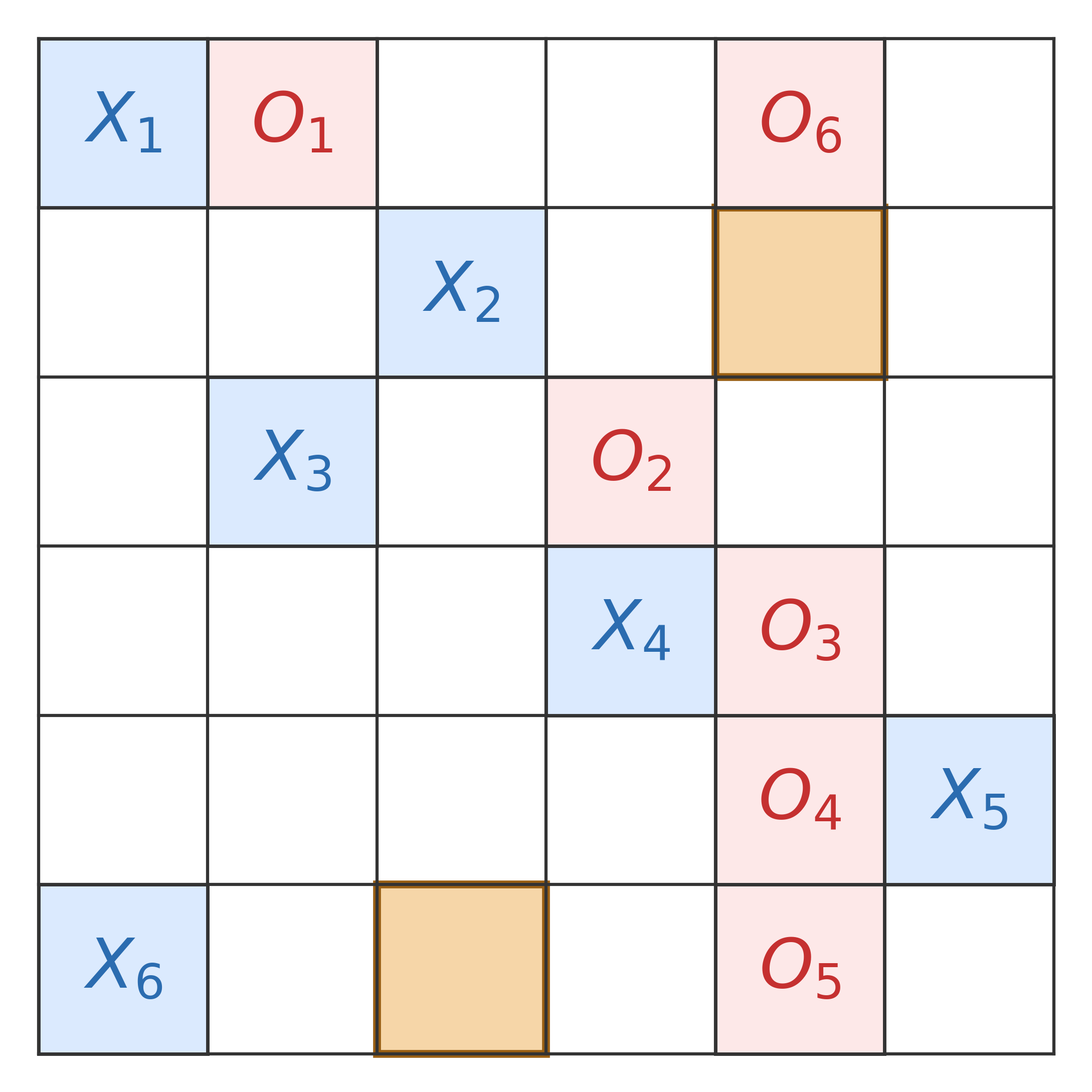}
\caption*{Ply 12: O blocks $(r,d)$}
\end{subfigure}
\end{figure}

\begin{figure}[htbp]
\centering
\begin{subfigure}[t]{0.30\textwidth}
\centering
\includegraphics[width=\textwidth]{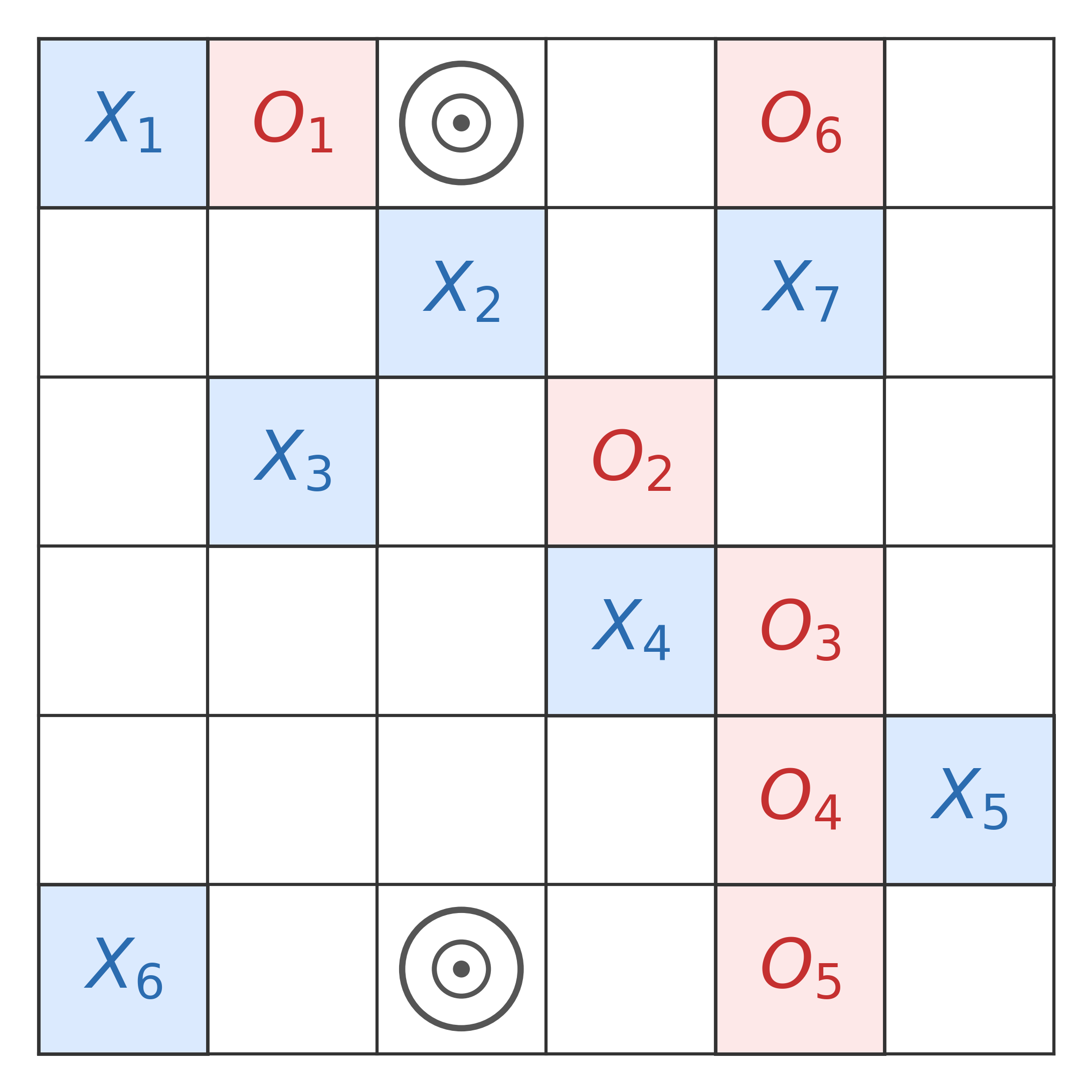}
\caption*{Ply 13: X plays $(s,d)=(2,5)$, creating double threat on $(b,\sigma(s))=(6,3)$ and $(r,\sigma(s))=(1,3)$}
\end{subfigure}
\hfill
\begin{subfigure}[t]{0.30\textwidth}
\centering
\includegraphics[width=\textwidth]{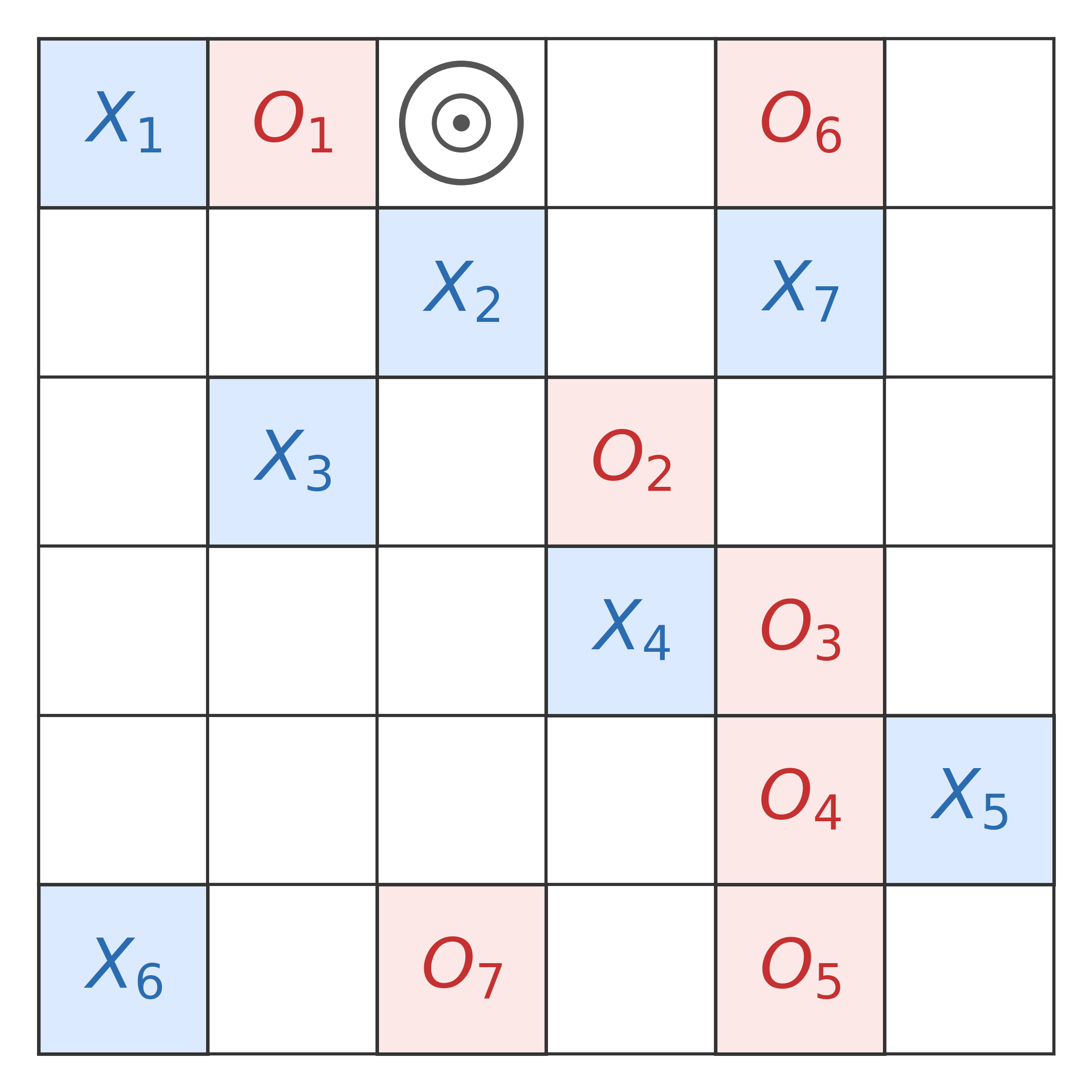}
\caption*{Ply 14: O blocks $(b,\sigma(s))$}
\end{subfigure}
\hfill
\begin{subfigure}[t]{0.30\textwidth}
\centering
\includegraphics[width=\textwidth]{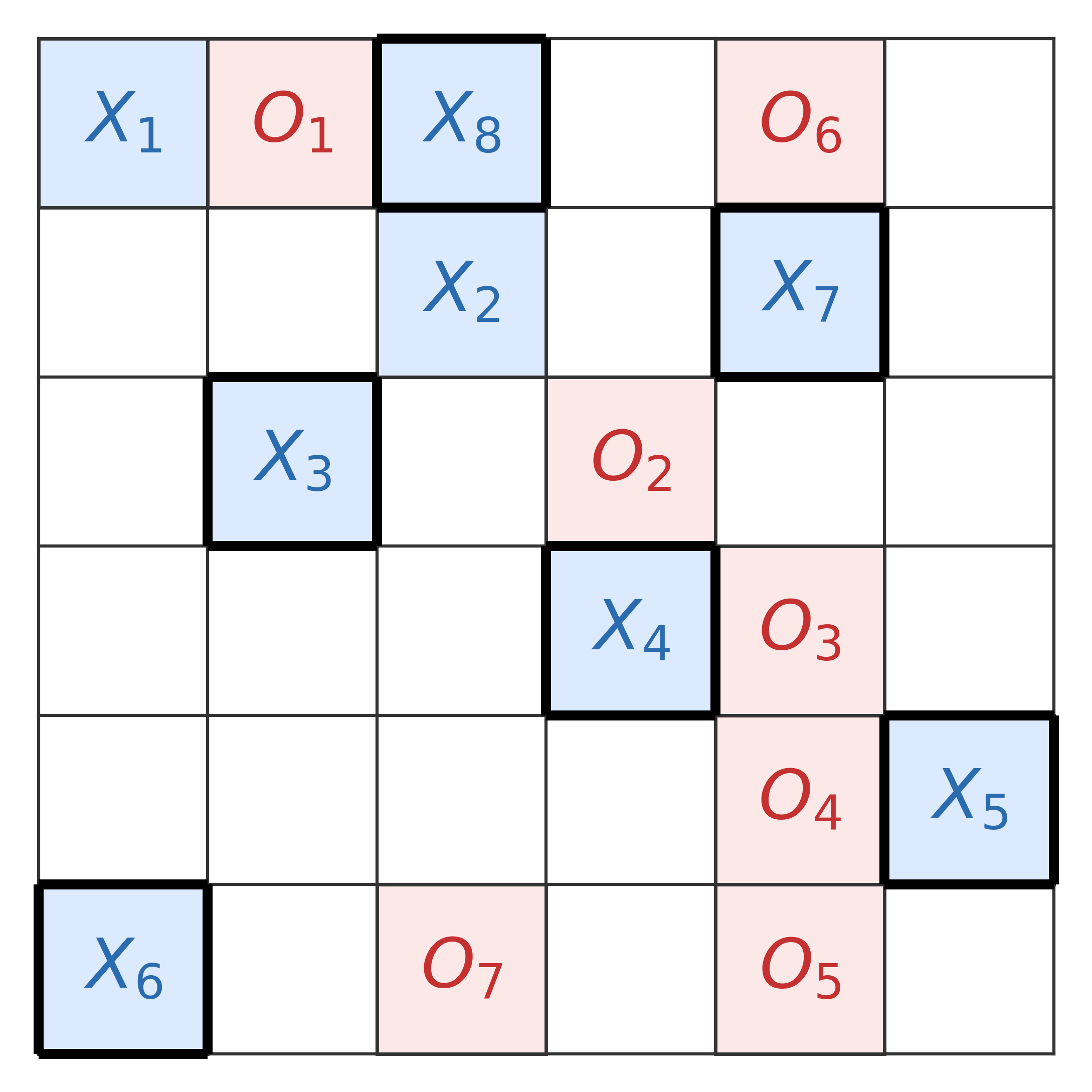}
\caption*{Ply 15: X plays $(r,\sigma(s))$ and completes transversal, winning the game}
\end{subfigure}
\end{figure}
\end{document}